\documentclass[11pt]{article}

\usepackage{xspace}
\usepackage{url}
\usepackage{mathtools}
\usepackage{amssymb}
\usepackage{amsthm}
\usepackage{empheq}
\usepackage{latexsym}
\usepackage{enumitem}
\usepackage{eurosym}
\usepackage{dsfont}
\usepackage{appendix}
\usepackage{color} 
\usepackage[unicode]{hyperref}
\usepackage{frcursive}
\usepackage[utf8]{inputenc}
\usepackage[T1]{fontenc}
\usepackage{geometry}
\usepackage{multirow}
\usepackage{todonotes}
\usepackage{lmodern}
\usepackage{anyfontsize} 
\usepackage{stmaryrd}
\usepackage{cleveref}
\usepackage[english]{babel}
\usepackage[english=british]{csquotes}
\usepackage{mathtools}
\usepackage{accents}
\usepackage{color}
\usepackage{comment}
\usepackage{hyperref}
\usepackage{bm}
\usepackage{algorithm}
\usepackage{algorithmicx}
\usepackage{algpseudocode}

\definecolor{red}{rgb}{0.7,0.15,0.15}
\definecolor{green}{rgb}{0,0.5,0}
\definecolor{blue}{rgb}{0,0,0.7}
\hypersetup{colorlinks, linkcolor={red}, citecolor={green}, urlcolor={blue}}

\allowdisplaybreaks

\usepackage{natbib}
\setcitestyle{numbers,open={[},close={]}}

\definecolor{red}{rgb}{0.7,0.15,0.15}
\definecolor{green}{rgb}{0,0.5,0}
\definecolor{blue}{rgb}{0,0,0.7}
\hypersetup{colorlinks, linkcolor={red}, citecolor={green}, urlcolor={blue}}
			
\makeatletter \@addtoreset{equation}{section}

\newtheorem{theorem}{Theorem}[section]
\newtheorem{assumption}[theorem]{Assumption}
\newtheorem{corollary}[theorem]{Corollary}
\newtheorem{example}[theorem]{Example}

\newtheorem{proposition}[theorem]{Proposition}

\newtheorem{definition}[theorem]{Definition}
\newtheorem{remark}[theorem]{Remark}

\usepackage{indentfirst}

\DeclareUnicodeCharacter{014D}{\=o}
\newcommand{\ba}{\begin{array}}
\newcommand{\ea}{\end{array}}
\newcommand{\be}{\begin{equation}}
\newcommand{\ee}{\end{equation}}
\newcommand{\bea}{\begin{eqnarray}}
\newcommand{\eea}{\end{eqnarray}}
\newcommand{\beaa}{\begin{eqnarray*}}
\newcommand{\eeaa}{\end{eqnarray*}}

\def\dbE{\mathbb{E}}

\def\dbR{\mathbb{R}}

\def\a{\alpha}
\def\b{\beta}

\def\d{\delta}
\def\e{\varepsilon}

\def\l{\lambda}

\def\n{\nu}
\def\si{\sigma}
\def\t{\tau}
\def\f{\varphi}
\def\th{\theta}

\def\D{\Delta}

\def\O{\Omega}
\def\cA{{\cal A}}
\def\cB{{\cal B}}

\def\cG{{\cal G}}

\def\cI{{\cal I}}

\def\cL{{\cal L}}
\def\cM{{\cal M}}

\def\cO{{\cal O}}

\def\cR{{\cal R}}

\def\cT{{\cal T}}

\def\q{\quad}
\def\pa{\partial}

\def\qed{ \hfill \vrule width.25cm height.25cm depth0cm\smallskip}

\def\bp{{\boldsymbol{p}}}
\def\bP{{\boldsymbol{P}}}

\def\1{\mathbf{1}}

\begin{document}

\title{Finite-player Optimal Stopping Games: Randomization, $\alpha$-potentiality, and Learning}
\author{Xin Guo\quad Mehdi Talbi \quad Qinxin Yan}

\maketitle

\begin{abstract}
Finite-player nonzero-sum optimal stopping games typically lead to coupled equilibrium systems whose complexity grows rapidly with the number of players. We introduce an independently randomized formulation in which each stopping rule is represented by an adapted, nondecreasing cumulative stopping process. The canonical embedding preserves pure-profile payoffs, and a pure profile is a Nash equilibrium of the original game if and only if its embedding is a Nash equilibrium of the randomized game. We adopt the $\alpha$-potential approach to  construct an $\alpha_N$-potential function, with the error $\alpha_N=O(N^{-1})$
under weak-interaction. We also identify an exact-potential subclass with a closed-form threshold equilibrium. For local stopped-status interactions, randomized payoffs admit a local stopped-mass representation, and potential maximization can be formulated as a multidimensional singular-control problem with local gradient constraints and a nonlocal condition for finite jumps. Under suitable regularity assumptions, we study the associated Hamilton–Jacobi–Bellman quasi-variational inequality and its regularity properties.  For unknown model coefficients, we propose a bounded-intensity Potential-CT-DDPG learning algorithm. Numerical experiments closely match the analytical benchmark and yield estimated best-response improvements consistent with $N^{-1}$ scaling.

\end{abstract}

\section{Introduction}
Nonzero-sum optimal stopping games arise in applications ranging from investment and market entry to adoption, abandonment, and resource-extraction decisions.  In such games,  the value of stopping for one player may depend on both the underlying dynamic system  and   other players'  stopping decisions. 
For  two-player optimal stopping games,  threshold-type Nash equilibria are
derived in \cite{DeAngelisFerrariMoriarty2018} by connecting equilibrium
characterization with coupled stopping and free-boundary
problems. $N$-player optimal stopping games are studied, for instance in
\cite{HamadeneMohammed2014,MartyrMoriarty2019,PossamaiTalbi2025}, where 
equilibrium is characterized via systems of coupled variational
inequalities, obstacle problems, or best-response problems.
 The dimension and analytical complexity of  optimal stopping games grow rapidly with the number of players, making equilibria difficult to characterize and compute.
 
One way to deal with this analytical  difficulty  is by the mean-field approach, which typically assumes player's homogeneity and their weak interactions  through the empirical distribution of the stopping times.  In such a framework, one class of mean-field optimal stopping games where interaction is mediated by the fraction of stopped agents is introduced and analyzed in \cite{Nutz2018}. Subsequently,  convergence from finite-player optimal stopping games to the mean-field game limit is investigated in \cite{NutzSanMartinTan2020}, an obstacle-problem approach is developed in \cite{Bertucci2020},  a linear programming formulation is studied in \cite{DumitrescuLeutscherTankov2021}, and  master-equation and the weak-equilibrium approaches are analyzed in \cite{PossamaiTalbi2025}, which also characterizes the existence of pure equilibria for the $N$-player game through a system of $N$ path-dependent obstacle problems. More recently,
dynamic programming and the viscosity method for mean-field optimal stopping  and its finite-population approximation have been developed in 
\cite{TalbiTouziZhang2023, TalbiTouziZhang2024}. 
 The benefit of the mean-field game approach is that, under suitable conditions, the Nash equilibrium of the mean-field game can be used to construct an approximate Nash equilibrium for the  original optimal stopping game. A major limitation of this approach is its reliance on restrictive assumptions concerning player homogeneity and the structure of their interactions.

\paragraph*{Our work.} This paper studies  finite-player optimal stopping games.
We adopt the framework of  $\alpha$-potential games, where  the change in any player’s payoff under a unilateral deviation is measured by the change of a single scalar functional called an $\alpha$-potential function. Consequently, Nash-equilibrium computation is reduced to  maximizing that functional, with an error bound $\alpha$. An  $\varepsilon$ maximizer of an $\alpha$-potential function yields an $\alpha+\varepsilon$-Nash equilibrium.  This dynamic game  approach \cite{GuoLiZhang2025,guo2026limittheorynplayeralphapotential}  relates  the error $\alpha$ to the asymmetry of the players’ mixed second-order variations. In particular,  sufficiently weak cross-player influences may yield an approximate Nash equilibrium whose error $\alpha_N$ decreases with the population size $N$.

To facilitate the $\alpha$-potential framework, we 
embed pure stopping times into a convex space of randomized stopping rules, following the classical compactification approach
called randomized stopping times
\cite{BaxterChacon1977}.  Intuitively, the players’ stopping times are generated using independent private randomization conditional on the common exogenous information. 
Technically, a randomized stopping time is represented by an adapted,
nondecreasing process $p^i$ taking values in $[0,1]$, with $p^i_t$ interpreted as the cumulative conditional probability  or cumulative stopping mass of player $i$ by time $t$. A pure
stopping time $\tau^i$ is then recovered through the embedding
$
p^i_t=\mathbf 1_{\{\tau^i\leq t\}}.
$
We will show that this randomization for optimal stopping games  preserves the payoff of every pure profile and, more strongly, a pure profile is an $\varepsilon$-Nash equilibrium of the pure stopping game if and only if its canonical embedding is an $\varepsilon$-Nash equilibrium of the randomized game (Proposition \ref{prop:euivalenceSG-RSG}). We emphasize that this randomization approach and these results hold for general optimal stopping games, beyond  the $\alpha$-potential framework.

Within the $\alpha$-potential game framework, we show that the canonical randomized extension of an $\alpha$-potential function for the pure game remains an $\alpha$-potential function, and its optimization introduces no relaxation gap (Proposition \ref{prop:equivalence-potential}).  We then identify sufficient conditions for $\alpha_N$ to be the order of $O(N^{-1})$, which include  games with dense interactions whose individual pairwise effects are of order $O(N^{-2})$, as well as  sparse interaction networks with uniformly bounded degree and edge effects of order $O(N^{-1})$.  
We also identify a subclass of  games whose payoff functions have a separation structure so that $\alpha_N=0$.
 We present an explicitly solvable example, where   the equilibrium  is derived by analyzing a collection of identical one-dimensional stopping problems (Section \ref{sec:exact_potential_example}).

However, this convexification alone does not make  general randomized optimal stopping games Markovian or local: after randomization, a payoff may still depend on the entire future stopping measures of other players. The stopped-status structure considered in this paper is therefore essential. Since player $i$'s pure payoff depends on the vector $(\mathbf{1}_{\{\tau^j\le \tau^i\}})_{j\ne i}$,  
  we can construct an $\alpha$-potential function, bounding the corresponding $\alpha_N$  by the aggregate asymmetry of the pairwise interactions (Propositions \ref{prop:local_alpha_estimate} and \ref{prop:local_influence_bound}). Consequently, an $\varepsilon$ maximizer of the potential function is an ($\alpha+\varepsilon$)-Nash equilibrium of the randomized optimal stopping game.
Moreover, 
the optimization problem for the potential function  is now translated into
a multidimensional singular-control problem. 
In other words, the convexification and the local representation enable both the potential construction and the subsequent dynamic-programming analysis.
Unlike standard singular-control models in which the marginal reward depends only on an exogenous state (see e.g. \cite{ cont2021interbank,
guo2019stochastic,soner1989regularity}), the marginal reward here depends  on the cumulative level of the singular controls through the stopped-mass vector.  The resulting
dynamic-programming equation is therefore a degenerate multidimensional
quasi-variational inequality combining a second-order operator in the
exogenous state, first-order gradient constraints in the stopping-mass
coordinates, and a nonlocal condition associated with finite jumps of the singular controls. 
For this class of  singular-control problems, we derive the Hamilton–Jacobi–Bellman quasi-variational inequality and investigate continuity, viscosity, and regularity properties of its value function, under the stated boundedness, Lipschitz, smoothness, and monotonicity assumptions. (Theorems \ref{thm:viscosity}, \ref{thm:regularity-infinite-horizon},  and~\ref{thm:verification}).

Finally, to learn the approximate Nash equilibria of general optimal stopping games when the model is unknown, we regularize the singular controls by imposing bounded, absolutely continuous stopping intensities, which converts the set of singular control into a compact continuous control set that are suitable for deterministic policy-gradient (DPG) methods. Building on the original DDPG method~\cite{LillicrapEtAl2016},  and, more specifically, the continuous-time CT-DDPG architecture of \cite{ChengGuoZhang2025} for continuous-time reinforcement learning, we propose a Potential-CT-DDPG algorithm consisting of a deterministic intensity actor, a value critic, a reparameterized advantage-rate critic, replay-buffer sampling, and target-network updates (Section \ref{sec:num_exp}). The numerical experiments indicate that the learned bounded-intensity policy approximates the analytical stopping benchmark. (Section \ref{sec:num_exp}). In the non-exact potential game at finite $N$,  the resulting empirical best-response gap scales approximately as $N^{-1}$, in agreement with the theoretical estimate for  $\alpha_N$ (Section \ref{sec:experiment-alpha})

\paragraph*{Related work.}

Our randomized optimal stopping game formulation differs  from \cite{dianetti2025entropy}, which reformulates a class of mean-field optimal stopping games using randomized stopping times and  adds entropy regularization to the resulting singular-control problem, thereby obtaining a  mixed, non-bang-bang stopping behavior. 
Recovering an equilibrium of the unregularized game in that framework would require a vanishing-entropy limit together with control of the regularization error. By contrast, our randomization serves as a convexification of the original finite-player game: the canonical embedding preserves pure-profile payoffs and pure $\varepsilon$-Nash equilibria, while the randomized extension of a potential introduces no relaxation gap. Moreover, we work directly at finite $N$, allow player-dependent reward and cost kernels, and do not first pass to a representative-agent mean-field limit as in \cite{Nutz2018}.

Our analysis adapts the dynamic $\alpha$-potential framework of \cite{GuoLiZhang2025} to finite-player optimal stopping games, where the nonconvexity of the pure stopping-time strategy space makes randomization essential. For stopped-mass interactions, the resulting potential problem takes the form of a multidimensional singular-control problem. Its dynamic-programming equation combines local gradient constraints governing continuous increases in stopping mass with an additional nonlocal condition for finite jumps, because the reward generated by an atom is evaluated at the pre-jump aggregate stopping profile. We are not aware of any existing singular-control formulation with precisely this local–nonlocal structure.
 

Algorithmically, our Potential-CT-DDPG  learning algorithm  is inspired by the deep deterministic policy-gradient methods in \cite{LillicrapEtAl2016}  for discrete-time continuous-control problems and in  \cite{ChengGuoZhang2025} for continuous-time reinforcement learning.
It combines the finite-player $\alpha_N$-potential reduction with the CT-DDPG architecture, thereby learning an approximate maximizer of the potential from simulated trajectories without requiring explicit knowledge of the state dynamics.  Unlike the entropy-regularized approach for mean-field optimal stopping games, 
our algorithm optimizes the bounded-intensity approximation of
the original potential problem without introducing an entropy penalty; it avoids integration over a continuous
action distribution in the policy update and
learns the Nash equilibrium for the original optimal stopping game.


\paragraph*{Organization of the paper.}
The remainder of the paper is organized as follows. Section \ref{sec:stopping-game} introduces the pure and randomized stopping games, and shows the consistency between the two formulations. Section \ref{sec:separation_game} presents an exact-potential separation structure and a closed-form example, which will be used as a benchmark for the  numerical experiments. Section~\ref{sec:stopped_mass_local_representation} defines a general formulation of stopping games with stopped-mass interactions, analyzes in detail the associated singular-control problem and studies the viscosity and the regularity properties of the associated  dynamic-programming equation  from the $\alpha$-potential framework.  Section \ref{sec: DDPG} introduces the Potential-CT-DDPG learning algorithm  when the model is unknown. 

\section{Stopping game setup}
\label{sec:stopping-game}

This section presents the mathematical formulation of the $N$-player stopping game.
Fix an integer $N\geq 2$. Let
$
(\Omega,\mathcal F,\mathbb F,\mathbb P),
\mathbb F=(\mathcal F_t^X)_{t\ge 0},
$
be a filtered probability space satisfying the usual conditions, generated from a continuous 
 exogenous state process
$
X=(X_t)_{t\ge 0}.
$
The filtration $\mathbb F$ represents the exogenous information
available to all players.

For every $t\ge 0$, define
$$
\mathcal T_{t}
:=
\left\{
\tau:\Omega\longrightarrow[t,\infty):
\tau\text{ is an }\mathbb F\text{-stopping time}
\right\},
$$
and write $\mathcal T:=\mathcal T_{0}$.

\subsection{The stopping game}
\label{sec:original_stopping_game}
For each $i\in[N]$, let
$$
G_i:
\bigl(\Omega\times[0,\infty)^N,
\mathcal F\otimes\mathcal B([0,\infty)^N)\bigr)
\longrightarrow
(\mathbb R,\mathcal B(\mathbb R))
$$
be jointly measurable. For a stopping profile
$\tau=(\tau^1,\ldots,\tau^N)\in\mathcal T^N$, define
$$
G_i(\tau)(\omega)
:=
G_i\bigl(
\omega,
\tau^1(\omega),
\ldots,
\tau^N(\omega)
\bigr).
$$
We assume that
$$
\sup_{\tau\in \mathcal T^N}\mathbb E^{\mathbb P}
\left[
\left|G_i(\tau)(\cdot)\right|
\right]
<\infty.
$$
The payoff of player $i$ under the pure stopping profile
$\tau$ is the deterministic functional
\begin{align}
\label{eq:pure_stopping_payoff}
J_i:\mathcal T^N\longrightarrow\mathbb R,
\qquad
J_i(\tau)
:=
\mathbb E^{\mathbb P}
\left[
G_i(\tau)(\cdot)
\right].
\end{align}
For
$
\tau^{-i}
:=
(\tau^1,\ldots,\tau^{i-1},\tau^{i+1},\ldots,\tau^N)
$
and $\sigma^i\in\mathcal T$, we use the convention
$
(\sigma^i,\tau^{-i})
:=
(\tau^1,\ldots,\tau^{i-1},
\sigma^i,
\tau^{i+1},\ldots,\tau^N).
$
Accordingly, given $\tau^{-i}$, the best-response value of player
$i$ is
\begin{align}
\label{eq:stopping_game}
    V_i(\tau^{-i})
:=
\sup_{\sigma^i\in\mathcal T}
J_i(\sigma^i,\tau^{-i}).
\tag{SG}
\end{align}
\begin{definition}[Pure $\varepsilon$-Nash equilibrium]
A profile
$\tau^\varepsilon\in\mathcal T^N$
is an $\varepsilon$-Nash equilibrium of the original stopping game~\eqref{eq:stopping_game}
if, for every $i\in[N]$,
$$
J_i(\tau^\varepsilon)
\geq
\sup_{\tau^i\in\mathcal T}
J_i(\tau^i,\tau^{\varepsilon,-i})
-\varepsilon.
$$
When $\varepsilon=0$, the profile is a Nash equilibrium.
\end{definition}
\begin{remark}[Game interaction]
    We use an open-loop information structure, such that under a unilateral
deviation by player $i$, the stopping rules of all other players
are kept fixed.
\end{remark}
\subsection{(Equivalent) randomized stopping game}
\label{sec: randomize_game}
Pure stopping times do not form a convex  space, so we consider  randomized stopping times to enlarge
the original game to a convex relaxed formulation so that each player
may distribute stopping mass over time. This relaxation retains every
pure stopping time through a canonical embedding, and we show below
that it preserves pure-profile payoffs and pure $\varepsilon$-Nash
equilibria. 

Specifically, let
$
(I,\mathcal I,\lambda)
:=
([0,1],\mathcal B([0,1]),\operatorname{Leb}).
$
Following Baxter and Chacon \cite{BaxterChacon1977}, a
randomized $\mathbb F$-stopping time is a map
$
\Theta:\Omega\times I\longrightarrow[0,\infty)
$
such that
$$
\{\Theta\leq t\}
\in
\mathcal F_t\otimes\mathcal I,
\qquad
0\leq t\leq T,
$$
and, for $\mathbb P$-a.e. $\omega$, the map
$
u\longmapsto\Theta(\omega,u)
$
is nondecreasing and left-continuous.

The cumulative stopping process associated with $\Theta$ is
$$
p_t(\omega)
:=
\lambda\left(
\left\{
u\in[0,1]:
\Theta(\omega,u)\leq t
\right\}
\right),
\qquad
t\ge 0,
$$
which is $\mathbb F$-adapted, c\`adl\`ag, nondecreasing, takes values
in $[0,1]$, and satisfies
$p_{0-}=0$ and $p_\infty=1$.
Conversely, any process with these properties generates a
randomized stopping time through
$$
\Theta^p(\omega,u)
:=
\inf\left\{
t\in[0,\infty):
p_t(\omega)>u
\right\},
$$
where the convention at $u=1$ is immaterial.

Now define
$$
\mathcal R
:=
\left\{
p:[0,\infty)\times\Omega\longrightarrow[0,1]:
\begin{array}{l}
p\text{ is }\mathbb F\text{-adapted, c\`adl\`ag,}\\
p\text{ is nondecreasing, }p_{0-}=0,\ p_\infty=1
\end{array}
\right\}.
$$
The set $\mathcal R$ is convex.

Let
$
(\overline\Omega,\overline{\mathcal F},\overline{\mathbb P})
:=
\left(
\Omega\times[0,1]^N,
\mathcal F\otimes\mathcal B([0,1]^N),
\mathbb P\otimes\lambda^{\otimes N}
\right),
$
and let $U^1,\ldots,U^N$ denote the coordinate maps on
$[0,1]^N$.
Given
$
p=(p^1,\ldots,p^N)\in\mathcal R^N$, 
define
$$
\Theta^{i,p^i}(\omega,U^i)
:=
\inf\left\{
t\ge0:
p_t^i(\omega)>U^i
\right\},
\qquad
i\in[N],
$$
and write
$$
\Theta^p
:=
\left(
\Theta^{1,p^1},\ldots,\Theta^{N,p^N}
\right).
$$
Consequently, the
randomized stopping times are conditionally independent given $\mathbb F$.

\begin{definition}[Canonical independently randomized extension]
For $p\in\mathcal R^N$, define
$$
J_i^R(p)
:=
\overline{\mathbb E}\left[
G_i(\Theta^p)(\cdot)
\right].
$$
Equivalently,
$$
J_i^R(p)
=
\int_{[0,1]^N}
J_i\left(
\Theta^{1,p^1}(\cdot,u_1),
\ldots,
\Theta^{N,p^N}(\cdot,u_N)
\right)
du_1\cdots du_N.
$$
\end{definition}

Given $p^{-i}\in\mathcal R^{N-1}$, the objective of player $i$ in this stopping game is to solve
\begin{align}
\label{eq:randomized_game}
    V_i^R(p^{-i})
:=
\sup_{q^i\in\mathcal R}
J_i^R(q^i,p^{-i}).
\tag{RSG}
\end{align}

\begin{definition}[Randomized $\varepsilon$-Nash equilibrium]
A profile
$p^\varepsilon\in\mathcal R^N$
is an $\varepsilon$-Nash equilibrium of the randomized game~\eqref{eq:randomized_game} if,
for every $i\in[N]$,
$$
J_i^R(p^\varepsilon)
\geq
\sup_{q^i\in\mathcal R}
J_i^R(q^i,p^{\varepsilon,-i})
-\varepsilon.
$$
When $\varepsilon=0$, the profile is a Nash equilibrium.
\end{definition}

Now we show the connections between the pure stopping game~\eqref{eq:stopping_game} and the randomized game~\eqref{eq:randomized_game}.
Define the canonical embedding
\begin{align*}
    \iota:\mathcal T^N
&\longrightarrow
\mathcal R^N,\\
\tau&\mapsto  (\iota(\tau)_t^i
:=
\mathbf 1_{\{\tau^i\leq t\}})_{i\in[N]}.
\end{align*}
Hence for Lebesgue-a.e. $u_i\in[0,1]$,
$
\Theta^{i,\iota(\tau)^i}(\omega,u_i)
=
\tau^i(\omega)$.
Consequently,
$$
J_i^R(\iota(\tau))
=
J_i(\tau).
$$
\begin{proposition}[Equivalence between pure and randomized stopping games]\label{prop:euivalenceSG-RSG}
For every $\varepsilon\geq0$ and every
$\tau^\varepsilon\in\mathcal T^N$,
$$
\tau^\varepsilon
\text{ is an $\varepsilon$-Nash equilibrium of the pure game~\eqref{eq:stopping_game}}
$$
if and only if
$$
\iota(\tau^\varepsilon)
\text{ is an $\varepsilon$-Nash equilibrium of the randomized game~\eqref{eq:randomized_game}}.
$$
\end{proposition}

\begin{proof}
Suppose first that $\tau^\varepsilon$ is an
$\varepsilon$-Nash equilibrium of the pure game. Fix $i\in[N]$ and
$q^i\in\mathcal R$. Since the opponents use pure strategies,
$$
J_i^R(q^i,\iota(\tau^{\varepsilon,-i}))
=
\int_0^1
J_i\left(
\Theta^{i,q^i}(\cdot,u_i),
\tau^{\varepsilon,-i}
\right)
du_i.
$$
For every fixed $u_i$, the section
$\Theta^{i,q^i}(\cdot,u_i)$ is an admissible pure stopping time.
Hence
$$
J_i\left(
\Theta^{i,q^i}(\cdot,u_i),
\tau^{\varepsilon,-i}
\right)
\leq
J_i(\tau^\varepsilon)+\varepsilon.
$$
Integrating over $u_i$ proves that
$\iota(\tau^\varepsilon)$ is an $\varepsilon$-Nash equilibrium of
the randomized game.

Conversely, every pure deviation is an admissible randomized
deviation through the embedding $\iota$. Thus a randomized
$\varepsilon$-Nash inequality at the pure profile
$\iota(\tau^\varepsilon)$ implies the corresponding pure
$\varepsilon$-Nash inequality.
\end{proof}
\begin{remark}
The preceding proposition concerns randomized equilibria that are
the embeddings of pure profiles. A general randomized Nash
equilibrium of the stopping game~\eqref{eq:randomized_game} need not be supported on pure stopping game Nash equilibria, even in
an exact potential game.
\end{remark}

\subsection{$\alpha$-potential games}

Although randomization yields a convex strategy space, computing a Nash
equilibrium directly still requires solving the players' coupled
best-response problems. We therefore seek an $\alpha$-potential
formulation, which approximates unilateral payoff variations by the
variations of a single scalar functional and thereby reduces
approximate-equilibrium computation to a single optimization problem.

We recall the $\alpha$-potential game framework introduced in
\cite{GuoLiZhang2025} and adapt it to the present stopping game.

\begin{definition}
[$\alpha$-potential 
function]
\label{def:alpha_potential}
Let
$
\mathcal A_i
\in
\left\{
\mathcal T,\mathcal R
\right\}
$ and $
\mathcal A:=\prod_{i=1}^N\mathcal A_i$,
and let $J_i:\mathcal A\to\mathbb R$ be the reward functional of
player $i$. A functional $\Phi:\mathcal A\to\mathbb R$ is called an
$\alpha$-potential function if, for every $i\in[N]$, every
$a=(a^i,a^{-i})\in\mathcal A$, and every unilateral deviation
$\widetilde a^i\in\mathcal A_i$,
$$
\left|
\left[
J_i(\widetilde a^i,a^{-i})
-
J_i(a^i,a^{-i})
\right]
-
\left[
\Phi(\widetilde a^i,a^{-i})
-
\Phi(a^i,a^{-i})
\right]
\right|
\leq\alpha.
$$
The game is then called an $\alpha$-potential game. When $\alpha=0$, the game is called an exact potential game and
$\Phi$ is called an exact potential function.
\end{definition}
\begin{proposition}[Potential maximizers and approximate equilibria]
Suppose that $\Phi$ is an $\alpha$-potential function and that
$a^\varepsilon\in\mathcal A$ satisfies
$$
\Phi(a^\varepsilon)
\geq
\sup_{a\in\mathcal A}\Phi(a)-\varepsilon.
$$
Then $a^\varepsilon$ is an $(\alpha+\varepsilon)$-Nash equilibrium.
\end{proposition}


\subsection{Consistency of $\alpha$-NE between pure and randomized stopping games}
The randomization in Section~\ref{sec: randomize_game} preserves not only the Nash equilibrium, but also the $\alpha$-potentiality.

\begin{proposition}[Randomization preserves $\alpha$-potentiality]
\label{prop:random_preserves_potential}
Suppose that
$
\Phi:\mathcal T^N\longrightarrow\mathbb R
$
is an $\alpha$-potential function for the original stopping game.
Define
\begin{align}
\label{eq:randomized_potential}
    \Phi^R(p)
:=
\int_{[0,1]^N}
\Phi\left(
\Theta^{1,p^1}(\cdot,u_1),
\ldots,
\Theta^{N,p^N}(\cdot,u_N)
\right)
du.
\end{align}
Then $\Phi^R$ is an $\alpha$-potential function for the canonical
randomized game. Conversely, if $\Phi^R$ is an $\alpha$-potential function for the
randomized game, then
$\Phi(\tau)
:=
\Phi^R(\iota(\tau))$ is an $\alpha$-potential function for the pure game.
\end{proposition}

\begin{proof}
Fix $i$, $p$, and a randomized unilateral deviation
$\widetilde p^i$. For every $u\in[0,1]^N$, the pure profiles
$
\Theta^{(\widetilde p^i,p^{-i}),u}$ and $
\Theta^{p,u}$
differ only in player $i$'s stopping time. Therefore, the pure
$\alpha$-potential inequality holds pointwise in $u$. Integrating
the inequality 
proves the first assertion. The converse follows by restricting
the randomized potential identity to embedded pure profiles.
\end{proof}
\begin{proposition}[No relaxation gap for a randomized potential]
\label{prop:equivalence-potential}
Let $\Phi^R$ be the randomized extension of $\Phi$ defined in~\eqref{eq:randomized_potential}.
Then
$$
\sup_{p\in\mathcal R^N}\Phi^R(p)
=
\sup_{\tau\in\mathcal T^N}\Phi(\tau).
$$
Moreover, a profile $p^\star$ maximizes $\Phi^R$ if and only if
$
\Theta^{p^\star,u}
\in
\operatorname{argmax}\Phi
$
for Lebesgue-a.e. $u\in[0,1]^N$.
\end{proposition}

\begin{proof}
Let
$
M
:=
\sup_{\tau\in\mathcal T^N}\Phi(\tau)$.
For every randomized profile $p$,
$$
\Phi^R(p)
=
\int_{[0,1]^N}
\Phi(\Theta^{p,u})\,du
\leq M.
$$
On the other hand, pure profiles are contained in the randomized
strategy space through $\iota$, and
$
\Phi^R(\iota(\tau))=\Phi(\tau)$. 
This proves equality of the suprema.
If $p^\star$ is a maximizer, then
$$
0
=
M-\Phi^R(p^\star)
=
\int_{[0,1]^N}
\left[
M-\Phi(\Theta^{p^\star,u})
\right]du.
$$
The integrand is nonnegative, and hence vanishes for a.e. $u$.
The converse is immediate.
\end{proof}

\begin{remark}[Role of independent randomization]
The preservation of $\alpha$-potentiality in
Proposition~\ref{prop:random_preserves_potential} does not require
the players' auxiliary randomizations to be independent. More
generally, the same conclusion holds if $\lambda^{\otimes N}$ is
replaced by any fixed probability measure $\kappa$ on $[0,1]^N$
with uniform marginals, independent of the exogenous probability
space. Indeed, for every fixed realization $u$, a unilateral
deviation by player $i$ changes only player $i$'s stopping time, so
the pure $\alpha$-potential inequality may be applied pointwise in
$u$ and then integrated with respect to $\kappa$.

We impose independent  randomization because the cumulative
rules $p^1,\ldots,p^N$ specify only the marginal conditional
stopping distributions, and the product coupling provides a
canonical joint law. Moreover, conditional independence is used
later in Proposition~\ref{prop:canonical_local_representation} to
derive the multi-affine local representation of the randomized
payoffs.
\end{remark}

\begin{remark}[General randomization does not imply locality]
\label{rem:general_randomization_not_local}

The Baxter--Chacon representation identifies a randomized stopping
time with its cumulative conditional distribution process. This
representation convexifies the strategy space, but it does not, by
itself, turn a general randomized stopping game into a local
singular-control problem.

Indeed, consider a general pure stopping payoff of the form
$$
J_i(\tau)
=
\mathbb E\left[
G_i\left(
\omega,\tau^1,\ldots,\tau^N
\right)
\right],
$$
where
$$
G_i:\Omega\times[0,\infty)^N\longrightarrow\mathbb R
$$
is a pathwise payoff kernel. Its canonical independently randomized
extension is
$$
J_i^R(p)
=
\mathbb E\left[
\int_{[0,\infty)^N}
G_i(\omega,t_1,\ldots,t_N)
\prod_{k=1}^N dp_{t_k}^k
\right].
$$
Equivalently, after integrating first with respect to the stopping
measure of player $i$, one may write
$$
J_i^R(p)
=
\mathbb E\left[
\int_{[0,
\infty)}
\mathcal G_i(\omega,t;p^{-i})
\,dp_t^i
\right],
$$
where
$$
\mathcal G_i(\omega,t;p^{-i})
:=
\int_{[0,\infty)^{N-1}}
G_i(\omega,t,t^{-i})
\prod_{j\neq i}dp_{t_j}^j.
$$
In general, the coefficient
$\mathcal G_i(\omega,t;p^{-i})$ depends on the entire stopping
measures
$$
dp^j|_{[0,\infty)},
\qquad j\neq i,
$$
including their restrictions to the future interval $(t,\infty)$.
Consequently, it need not be determined by the current cumulative
stopping vector $p_t$, and it need not admit an adapted Markovian
representation of the form
$$
\mathcal G_i(\omega,t;p^{-i})
=
F_i(t,X_t,p_t).
$$
\end{remark}

\subsection{Exact potentiality under a separation structure}
\label{sec:separation_game}
In this section, we introduce a class of optimal stopping games that are  exact potential games, that is, a game with $\alpha=0$.
\subsubsection{Optimal stopping game and its randomization}
Retain the setup of the  optimal stopping game in Section~\ref{sec:original_stopping_game}, and further assume the following separation structure for the payoff function.
\begin{assumption}[Separation structure]\label{ass:pure-separation}
There exist a jointly measurable common payoff kernel
$$
    \mathcal H:
    \Omega\times[0,\infty)^N\rightarrow\mathbb R
$$
and, for every $i\in[N]$, a jointly measurable residual kernel
$$
    \mathcal U_i:
    \Omega\times[0,\infty)^{N-1}\rightarrow\mathbb R
$$
such that, for $\mathbb P$-a.e.\ $\omega$ and every
$\mathbf t=(t_1,\ldots,t_N)\in[0,\infty)^N$,
$$
    G_i(\omega,\mathbf t)
    =
    \mathcal H(\omega,\mathbf t)
    +
    \mathcal U_i(\omega,\mathbf t^{-i}).
$$
In particular, $\mathcal U_i$ does not depend on player $i$'s stopping time $t_i$.
Moreover, for every
$\boldsymbol{\tau}\in\mathcal T^N_0$ and every $i\in[N]$,
$$
    \mathbb E\left[
        \left|
            \mathcal H(\omega,\boldsymbol{\tau})
        \right|
        +
        \left|
            \mathcal U_i(\omega,\boldsymbol{\tau}^{-i})
        \right|
    \right]
    <\infty.
$$
\end{assumption}

\begin{proposition}[Exact potential for the pure stopping game]
\label{prop:pure-exact-potential}
Suppose Assumption~\ref{ass:pure-separation} holds. Define
$$
    \Phi(\boldsymbol{\tau})
    :=
    \mathbb E\left[
        \mathcal H
        \left(
            \omega,
            \tau^1,\ldots,\tau^N
        \right)
    \right].
$$
Then $\Phi$ is an exact potential function for the pure stopping game. More precisely,
for every $i\in[N]$, every
$\boldsymbol{\tau}\in(\mathcal T_{0})^N$, and every unilateral deviation
$\widetilde\tau^i\in\mathcal T_{0}$,
$$
\begin{aligned}
&
J_i
\left(
    \widetilde\tau^i,\boldsymbol{\tau}^{-i}
\right)
-
J_i
\left(
    \tau^i,\boldsymbol{\tau}^{-i}
\right)=
\Phi
\left(
    \widetilde\tau^i,\boldsymbol{\tau}^{-i}
\right)
-
\Phi
\left(
    \tau^i,\boldsymbol{\tau}^{-i}
\right).
\end{aligned}
$$
Consequently, the pure stopping game is an exact potential game, that is,
$\alpha=0.$
In particular, every maximizer of $\Phi$ is a Nash equilibrium of the pure stopping game.
\end{proposition}

\begin{proof}
By Assumption~\ref{ass:pure-separation},
$$
\begin{aligned}
J_i
\left(
    \widetilde\tau^i,\boldsymbol{\tau}^{-i}
\right)
&=
\mathbb E\left[
    \mathcal H
    \left(
        \omega,
        \widetilde\tau^i,
        \boldsymbol{\tau}^{-i}
    \right)
    +
    \mathcal U_i
    \left(
        \omega,
        \boldsymbol{\tau}^{-i}
    \right)
\right],
\\
J_i
\left(
    \tau^i,\boldsymbol{\tau}^{-i}
\right)
&=
\mathbb E\left[
    \mathcal H
    \left(
        \omega,
        \tau^i,
        \boldsymbol{\tau}^{-i}
    \right)
    +
    \mathcal U_i
    \left(
        \omega,
        \boldsymbol{\tau}^{-i}
    \right)
\right].
\end{aligned}
$$
The residual term is unchanged under a unilateral deviation by player $i$. Hence 
$$
\begin{aligned}
J_i
\left(
    \widetilde\tau^i,\boldsymbol{\tau}^{-i}
\right)
-
J_i
\left(
    \tau^i,\boldsymbol{\tau}^{-i}
\right)&=
\mathbb E\left[
    \mathcal H
    \left(
        \omega,
        \widetilde\tau^i,
        \boldsymbol{\tau}^{-i}
    \right)
    -
    \mathcal H
    \left(
        \omega,
        \tau^i,
        \boldsymbol{\tau}^{-i}
    \right)
\right]
=
\Phi
\left(
    \widetilde\tau^i,\boldsymbol{\tau}^{-i}
\right)
-
\Phi
\left(
    \tau^i,\boldsymbol{\tau}^{-i}
\right).
\end{aligned}
$$
This proves the exact potentiality.
\end{proof}

\begin{remark}[Symmetric pairwise separation structure]
A sufficient condition for Assumption~\ref{ass:pure-separation} is
$$
    \mathcal G_i(\omega,\mathbf t)
    =
    c_i(\omega,t_i)
    +
    \sum_{j\neq i}
    h_{ij}(\omega,t_i,t_j)
    +
    r_i(\omega,\mathbf t^{-i}),
$$
where $r_i$ does not depend on $t_i$ and the pairwise interaction kernels satisfy
$$
    h_{ij}(\omega,s,t)
    =
    h_{ji}(\omega,t,s)
$$
for every $i\neq j$. In this case, a common potential kernel is
$$
    \mathcal H(\omega,\mathbf t)
    =
    \sum_{k=1}^N
    c_k(\omega,t_k)
    +
    \sum_{1\leq k<\ell\leq N}
    h_{k\ell}(\omega,t_k,t_\ell).
$$
\end{remark}

Now, following the definition of randomized stopping times and randomized optimal stopping games in Section~\ref{sec: randomize_game}, 
 define the randomized payoff function
$$
     J_i^R(\boldsymbol p)
    :=
    \overline{\mathbb E}\left[
         G_i
        \left(
            \omega,
            \boldsymbol{\Theta}^{\boldsymbol p}
        \right)
    \right],
$$
and a randomized potential function
$$
 \Phi^R(\boldsymbol p)
    :=
    \overline{\mathbb E}\left[
        \mathcal H
        \left(
            \omega,
            \boldsymbol{\Theta}^{p}
        \right)
    \right]=
    \int_{[0,1]^N}
    \Phi
    \left(
        \Theta^{1,p^1}(\cdot,u_1),
        \ldots,
        \Theta^{N,p^N}(\cdot,u_N)
    \right)
    du_1\cdots du_N.
$$
The following result directly follows from Proposition~\ref{prop:random_preserves_potential} with $\alpha=0$.
\begin{corollary}[Randomization preserves exact potentiality]
\label{prop:randomization-preserves-exact-potential}
Suppose Assumption~\ref{ass:pure-separation} holds. 
Then $\Phi^R$ is an exact potential function for the randomized stopping game. 
In particular, every maximizer of $\Phi^R$ is a Nash equilibrium of the randomized stopping
game, and
$$
\begin{aligned}
\arg\max\Phi^R
=
\Bigl\{
\boldsymbol p\in(\mathcal R)^N:
\boldsymbol{\Theta}^{p,u}
\in
\arg\max\Phi
\text{ for a.e.\ }u
\Bigr\}.
\end{aligned}
$$
\end{corollary}
\subsubsection{An example with explicit solution}
\label{sec:exact_potential_example}
Furthermore, we provide an example of an exact potential game with a closed-form solution.

Consider the $N$-player potential game with 
$$
J_i(\tau_1,\dots,\tau_N)
=
\dbE_x\!\left[e^{-\rho \tau_i}(X_{\tau_i}-K)^+\right]
+
\frac{\eta}{N-1}\sum_{j\neq i}
\dbE_x\!\left[e^{-\rho \tau_j}(X_{\tau_j}-K)^+\right],
\qquad i=1,\dots,N,
$$
where $\eta\ge 0$, $\rho>0$, $K>0$, and the state process is
$
X_t=x+\sigma B_t,
$
with $\sigma>0$ and $B$ a standard Brownian motion. Each player chooses a stopping time $\tau_i\in \cT$.

\begin{proposition}
The above stopping game is an exact potential game with potential function
$$
\Phi(\tau_1,\dots,\tau_N)
=
\sum_{i=1}^N
\dbE_x\!\left[e^{-\rho \tau_i}(X_{\tau_i}-K)^+\right].
$$
Moreover, a Nash equilibrium is given by
$$
\tau_1^*=\cdots=\tau_N^*=\tau^*,
\qquad
\tau^*:=\inf\{t\ge 0:\ X_t\ge x^*\},
$$
where the optimal stopping boundary is
$$
x^* = K + \frac{\sigma}{\sqrt{2\rho}}.
$$
The corresponding single-agent value function is
$$
V_{\mathrm{exact}}(x)
=
\sup_{\tau}\dbE_x\!\left[e^{-\rho \tau}(X_\tau-K)^+\right]
=
\begin{cases}
x-K, & x\ge x^*,\\[4pt]
(x^*-K)\exp\!\left(\dfrac{\sqrt{2\rho}}{\sigma}(x-x^*)\right), & x<x^*.
\end{cases}
$$
Hence the equilibrium potential value is
$$
\Phi^*(x)=N\,V_{\mathrm{exact}}(x).
$$
\end{proposition}
\begin{proof}
For each stopping time $\tau$, define the single-agent payoff
$$
F(\tau)
:=
\mathbb E_x\left[e^{-\rho \tau}(X_\tau-K)^+\right].
$$
Then the payoff of player $i$ can be written as
$
J_i(\tau_1,\ldots,\tau_N)
=
F(\tau_i)
+
\frac{\eta}{N-1}\sum_{j\ne i}F(\tau_j).
$
Let player $i$ deviate unilaterally from $\tau_i$ to $\tau_i'$, while the stopping times of the other players are fixed. Since the interaction term
$
\frac{\eta}{N-1}\sum_{j\ne i}F(\tau_j)
$
does not depend on $\tau_i$, we have
$
J_i(\tau_i',\tau_{-i})-J_i(\tau_i,\tau_{-i})
=
F(\tau_i')-F(\tau_i).
$    
Meanwhile, for
$
\Phi(\tau_1,\ldots,\tau_N)
:=
\sum_{k=1}^N F(\tau_k),
$
we also have
$
\Phi(\tau_i',\tau_{-i})-\Phi(\tau_i,\tau_{-i})
=
F(\tau_i')-F(\tau_i).
$
Therefore,
$$
J_i(\tau_i',\tau_{-i})-J_i(\tau_i,\tau_{-i})
=
\Phi(\tau_i',\tau_{-i})-\Phi(\tau_i,\tau_{-i}),
$$
for every $i\in[N]$ and every unilateral deviation. Hence the game is an exact potential game with potential function $\Phi$.

To solve the single-agent optimal stopping problem, we first define a candidate solution, and then verify optimality. Let
$$
v(x)
:=
\begin{cases}
x-K, & x\ge x^*,\\[4pt]
(x^*-K)\exp\left(\lambda(x-x^*)\right), & x<x^*,
\end{cases}
$$
where $x^*:=K+\frac{\sigma}{\sqrt{2\rho}}$. We will show that $v(x)=V_{\mathrm{exact}}(x)$.

First, $v$ is nonnegative, $C^1$, and twice differentiable away from $x^*$. Moreover,
$$
v(x)\ge (x-K)^+,
\qquad x\in\mathbb R.
$$
Furthermore, following direct calculations, we can see that in the continuation region $x<x^*$,
$$
\frac{\sigma^2}{2}v''(x)-\rho v(x)=0,
$$
while in the stopping region $x>x^*$,
$$
\frac{\sigma^2}{2}v''(x)-\rho v(x)
=
-\rho(x-K)
\le 0.
$$
Applying Itô's formula to $e^{-\rho t}v(X_t)$ and localizing if necessary, we obtain that
$
\left(e^{-\rho t}v(X_t)\right)_{t\ge 0}
$
is a supermartingale. Therefore, for every stopping time $\tau$,
$$
v(x)
\ge
\mathbb E_x\left[e^{-\rho \tau}v(X_\tau)\right]
\ge
\mathbb E_x\left[e^{-\rho \tau}(X_\tau-K)^+\right].
$$
Taking the supremum over $\tau$ gives
$$
v(x)\ge V_{\mathrm{exact}}(x).
$$
Now define
$$
\tau^*
:=
\inf\{t\ge 0:X_t\ge x^*\}.
$$
If $x\ge x^*$, then $\tau^*=0$ and
$$
\mathbb E_x\left[e^{-\rho\tau^*}(X_{\tau^*}-K)^+\right]
=
x-K
=
v(x).
$$
If $x<x^*$, then $X_{\tau^*}=x^*$ and
$$
\mathbb E_x\left[e^{-\rho\tau^*}(X_{\tau^*}-K)^+\right]
=
(x^*-K)\mathbb E_x\left[e^{-\rho\tau^*}\right].
$$
Since $\lambda=\sqrt{2\rho}/\sigma$, the process
$
\exp\left(\lambda(X_t-x^*)-\rho t\right)
$
is a martingale. Optional stopping at $\tau^*$ gives
$$
\mathbb E_x\left[e^{-\rho\tau^*}\right]
=
\exp\left(\lambda(x-x^*)\right).
$$
Therefore,
$$
\mathbb E_x\left[e^{-\rho\tau^*}(X_{\tau^*}-K)^+\right]
=
(x^*-K)\exp\left(\lambda(x-x^*)\right)
=
v(x).
$$
Thus $\tau^*$ is optimal and $V_{\mathrm{exact}}=v$.

Finally, since player $i$'s best-response problem is
$$
\sup_{\tau_i} J_i(\tau_i,\tau_{-i})
=
\sup_{\tau_i}
\left[
F(\tau_i)
+
\frac{\eta}{N-1}\sum_{j\ne i}F(\tau_j)
\right],
$$
and the second term is independent of $\tau_i$, every player has the same best response $\tau^*$. Hence
$
\tau_1^*=\cdots=\tau_N^*=\tau^*
$
is a Nash equilibrium, which yields
$
\Phi^*(x)
=
\sum_{i=1}^N F(\tau^*)
=
N V_{\mathrm{exact}}(x).
$
\end{proof}

\section{Stopping games with stopped-mass interactions}
\label{sec:stopped_mass_local_representation}
We now specify the general formulation of  stopping games to games with stopped-mass interactions, which will be analyzed in detail under the $\alpha$-potential framework.

For every player $i\in[N]$, let
$$
h_i:
\mathbb R^d\times[0,\infty)\times\{0,1\}^{N-1}
\to\mathbb R
$$
be a measurable reward function.  Recall that $\mathbb F$ is generated from a continuous 
 exogenous state process
$
X=(X_t)_{t\ge 0}$. For a pure stopping profile
$\tau\in\mathcal T^N$, define
\begin{align}
\label{eq:stopped_mass_game}
    J_i(\tau)
:=
\mathbb E
\left[
h_i
\left(
X_{\tau^i},
\tau^i,
\left(
\mathbf 1_{\{\tau^j\leq\tau^i\}}
\right)_{j\neq i}
\right)
\right].
\tag{SMG}
\end{align}
For $q^{-i}=(q^j)_{j\neq i}\in[0,1]^{N-1}$, $a^{-i}=(a^j)_{j\neq i}\in \{0,1\}^{N-1}$, define the multi-affine
extension of $h_i$ by
\begin{align}
\label{eq: g_local_form}
    g_i^N(x,t,q^{-i})
:=
\sum_{a^{-i}\in\{0,1\}^{N-1}}
h_i(x,t,a^{-i})
\prod_{j\neq i}
(q^j)^{a^j}(1-q^j)^{1-a^j}.
\end{align}
\begin{proposition}
\label{prop:canonical_local_representation}
For the stopped-mass game~\eqref{eq:stopped_mass_game}, the canonical randomized payoff
admits the local representation
\begin{equation}
\label{eq:local_payoff}
    J_i^R(p)
=
\mathbb E
\left[
\int_{[0,\infty)}
g_i^N(X_t,t,p_t^{-i})\,dp_t^i
\right].
\end{equation}
\end{proposition}

\begin{proof}
Conditionally on $\mathbb F$ and on the private randomization
$U^i$ of player $i$, set
$
t:=\Theta^{i,p^i}$. For every $j\neq i$,
$$
\mathbf 1_{\{\Theta^{j,p^j}\leq t\}}
=
\mathbf 1_{\{U^j<p_t^j\}}.
$$
The variables on the right-hand side are conditionally independent
Bernoulli random variables with parameters $(p_t^j)_{j\neq i}$.
Therefore,
$$
\begin{aligned}
&
\mathbb E_{\overline{\mathbb P}}
\left[
h_i
\left(
X_t,t,
\left(
\mathbf 1_{\{\Theta^{j,p^j}\leq t\}}
\right)_{j\neq i}
\right)
\Bigm|
\mathbb F,U^i
\right]
=
g_i^N(X_t,t,p_t^{-i}).
\end{aligned}
$$

Finally, conditionally on $\mathbb F$, the distribution of
$\Theta^{i,p^i}$ is the Stieltjes measure $dp^i$, which gives the
claimed representation.
\end{proof}
\begin{remark}[Filtration and stopped-status interactions]
\label{rem:filtration_stopped_status}

The payoff~\eqref{eq:stopped_mass_game} does not require a change
of the admissible filtration in the open-loop formulation considered
here. 

For a pure stopping profile
$\tau\in\mathcal T^N$, the process
$
N_t^j:=\mathbf 1_{\{\tau^j\leq t\}}$
is $\mathbb F$-adapted and càdlàg. Consequently,
$N_{\tau^i}^j
=
\mathbf 1_{\{\tau^j\leq\tau^i\}}$
is $\mathcal F_{\tau^i}$-measurable. Thus the payoff
$h_i(
X_{\tau^i},
\tau^i,
(\mathbf 1_{\{\tau^j\leq\tau^i\}}
)_{j\neq i})$
is well defined under the original filtration.

This information structure means that the realized stopped-status
profile affects the payoff, but the cumulative rule $p^i$ cannot be
revised in response to the other players' private randomizations.
If players are instead allowed to observe and react to the realized
stopping decisions of the other players, the admissible filtration
must be enlarged to ensure measurability. 
\end{remark}

\subsection{$\alpha$-potential function on the randomized
strategy space}

\label{subsec:randomized_alpha_potential}
We now construct an $\alpha$-potential function  introduced in~\cite{GuoLiZhang2025}.

For a functional $F:\mathcal R^N\to\mathbb R$, its linear
derivative with respect to the action of player $i$ is a map
$$
\frac{\delta F}{\delta p^i}:
\mathcal R^N
\times
\operatorname{span}(\mathcal R)
\to
\mathbb R
$$
such that, for every $p\in\mathcal R^N$ and
$q^i\in\mathcal R$,

$$
\lim_{\varepsilon\downarrow0}
\frac{
F(p^i+\varepsilon(q^i-p^i),p^{-i})-F(p)
}{
\varepsilon
}
=
\frac{\delta F}{\delta p^i}
\left(
p;q^i-p^i
\right).
$$

Assuming that differentiation may be interchanged with the expectation
and the Stieltjes integral, then in~\eqref{eq:local_payoff}, we have
$$
\frac{\delta J_i^R}{\delta p^i}
\left(
p;\eta^i
\right)
=
\mathbb E
\left[
\int_{[0,\infty)}
g_i^N(X_t,t,p_t^{-i})\,d\eta_t^i
\right].
$$
If $j\neq i$, then
$$
\frac{\delta^2 J_i^R}
{\delta p^i\delta p^j}
\left(
p;\eta^i,\eta^j
\right)
=
\mathbb E
\left[
\int_{[0,\infty)}
\partial_j g_i^N(X_t,t,p_t^{-i})
\eta_t^j\,d\eta_t^i
\right],
$$
where $\partial_j g_i^N$ denotes the derivative with respect to the
coordinate of player $j$. Since $g_i^N$ does not depend on $p^i$,
$$
\frac{\delta^2 J_i^R}
{\delta p^i\delta p^i}
=0.
$$
Fix a deterministic base profile
$
z=(z^1,\ldots,z^N)\in\mathcal R^N$.
Define
$$
\Phi_z^R(p)
:=
\int_0^1
\sum_{i=1}^N
\frac{\delta J_i^R}{\delta p^i}
\left(
z+r(p-z);p^i-z^i
\right)
dr.
$$
For the local randomized stopping game, this becomes
\begin{align}
\label{eq:local_potential}
    \Phi_z^R(p)
=
\int_0^1
\sum_{i=1}^N
\mathbb E
\left[
\int_{[0,\infty)}
g_i^N
\left(
X_t,t,
\left\{
z_t+r(p_t-z_t)
\right\}^{-i}
\right)
d(p_t^i-z_t^i)
\right]
dr.
\end{align}
\begin{remark}[Choice of the base point]
\label{rem:base_point_zero}
 A natural base point  is the pure rule that "never stops at finite time". To define a base point corresponding to never stopping, we compactify the
time domain as
$
\overline{\mathbb R}_+ := [0,\infty].
$
Let $\overline{\mathcal T}$ denote the set of
$\overline{\mathbb R}_+$-valued stopping times, where
$\tau=\infty$ is interpreted as never stopping. We assign zero payoff when
player $i$ never stops.

We introduce the ambient randomized-strategy space
$$
\overline{\mathcal R}
:=
\left\{
p:[0,\infty)\times\Omega\to[0,1]:
p\text{ is adapted, càdlàg, nondecreasing, and }p_{0-}=0
\right\}.
$$
For $p\in\overline{\mathcal R}$, define
$
p_{\infty-}:=\lim_{t\to\infty}p_t,
$
and associate with $p$ the random probability measure $\mu^p$ on
$\overline{\mathbb R}_+$ given by
$$
\mu^p([0,t]):=p_t,\quad \forall\,t<\infty,\qquad
\mu^p(\{\infty\}):=1-p_{\infty-}.
$$

The original randomized-strategy space is the subset
$
\mathcal R
=
\left\{
p\in\overline{\mathcal R}:p_{\infty-}=1
\right\}.
$ Let $\tau^\infty\equiv\infty$ and define its canonical embedding
$z^\infty:=\iota(\tau^\infty)$. Thus,
$$
z_t^\infty=0,\quad \forall\,t<\infty,\qquad
\mu^{z^\infty}=\delta_\infty.
$$
For every $p\in\overline{\mathcal R}$ and $r\in[0,1]$,
$$
\mu^{z^\infty+r(p-z^\infty)}
=
(1-r)\delta_\infty+r\mu^p,
$$
so that, at each finite time,
$$
\left(z^\infty+r(p-z^\infty)\right)_t=rp_t.
$$

Extend the marginal reward by setting its value at the cemetery time equal
to zero. Then the line-integral potential based at $z^\infty$ is
$$
\Phi_{z^\infty}^R(p)
=
\int_0^1
\sum_{i=1}^N
\mathbb E
\left[
\int_{\overline{\mathbb R}_+}
\overline g_i^N
\left(
X_t,t,
\left\{
z_t^\infty+r(p_t-z_t^\infty)
\right\}^{-i}
\right)
\left(\mu^{p^i}-\delta_\infty\right)(dt)
\right]
dr.
$$
Since the reward at $\infty$ is zero, this reduces to
$$
\Phi_{z^\infty}^R(p)
=
\sum_{i=1}^N
\mathbb E
\left[
\int_{[0,\infty)}
G_i(X_t,t,p_t^{-i})\,dp_t^i
\right],
$$
where
$$
G_i(x,t,q^{-i})
:=
\int_0^1 g_i^N(x,t,rq^{-i})\,dr.
$$
We construct the potential on $\overline{\mathcal R}^N$ and then, when
desired, restrict it to the original strategy space $\mathcal R^N$.
\end{remark}

For $i,j\in [N]$ and every $(x,t,q)\in \mathbb R^d\times [0,\infty)\times \mathcal R^N$, define
\begin{align}
\label{eq:a_ij}
    a_{ij}^N(x,t,q)
:=
\partial_{q^j}g_i^N(x,t,q^{-i}),
\end{align}
and for every $p\in \mathcal R^N$,
$$
\begin{aligned}
\mathfrak C_{ij}^N
\left(
p;q^i,q^j
\right)
:=
\mathbb E\left[
\int_{[0,\infty)}
a_{ij}^N(X_t,t,p_t)
q_t^j\,dq_t^i
\right]-
\mathbb E\left[
\int_{[0,\infty)}
a_{ji}^N(X_t,t,p_t)
q_t^i\,dq_t^j
\right].
\end{aligned}
$$
Set
$$
\Gamma_N
:=
\sup_{\substack{
i\in[N],\\
p,q\in\mathcal R^N
}}
\sum_{j\neq i}
\left|
\mathfrak C_{ij}^N
\left(
p;q^i,q^j
\right)
\right|.
$$
Recall the definition of an  $\alpha$-potential function in Definition~\ref{def:alpha_potential}. The following theorem follows \cite{GuoLiZhang2025}.
\begin{proposition}
[$\alpha_N$-potential function]
\label{prop:local_alpha_estimate}
$\Phi_z^R$ is an $\alpha_N$-potential function for the randomized
stopping game with
$$
\alpha_N\leq 2\Gamma_N.
$$
Consequently, if $p^\varepsilon$ satisfies
$$
\Phi_z^R(p^\varepsilon)
\geq
\sup_{p\in\mathcal R^N}\Phi_z^R(p)-\varepsilon,
$$
then $p^\varepsilon$ is an
$(\alpha_N+\varepsilon)$-Nash equilibrium.
\end{proposition}

Next we provide some sufficient conditions such that $\Gamma_N=O(N^{-1})$. Recall that $g^N$ admits the representation~\eqref{eq: g_local_form}. For $i,j\in [N]$ and $a^{-i,j}=(a^k)_{k\neq i,j}\in\{0,1\}^{N-2}$, define
$$
\Delta_j h_i^N(x,t,a^{-i,j})
:=
h_i^N(x,t,1,a^{-i,j})-
h_i^N(x,t,0,a^{-i,j}),\qquad L_{ij}^N
:=
\sup_{\substack{
x\in\mathbb R^d,\,
t\ge 0,\\
a^{-i,j}\in\{0,1\}^{N-2}
}}
\left|
\Delta_jh_i^N(x,t,a^{-i,j})
\right|.
$$
\begin{proposition}
\label{prop:local_influence_bound}
For every integer $N\ge 2$,
$$
\alpha_N
\leq
2
\max_{i\in[N]}
\sum_{j\neq i}
\left(
L_{ij}^N+L_{ji}^N
\right).
$$
\end{proposition}
\begin{proof}
    By the definition of $a_{ij}$~\eqref{eq:a_ij},
$ \left|
a_{ij}^N(x,t,q)
\right|
\leq L_{ij}^N$. Moreover, for every $q^i,q^j\in\mathcal R$,
$$
0\leq q_t^i,q_t^j\leq1,
\qquad
\int_0^\infty dq_t^i
=
\int_0^\infty dq_t^j
=
1.
$$
It follows that
$$
\left|
\mathfrak C_{ij}^N
\left(
p;q^i,q^j
\right)
\right|
\leq
L_{ij}^N+L_{ji}^N.
$$
Consequently, Proposition ~\ref{prop:local_alpha_estimate} gives
$$
\alpha_N
\leq
2
\max_{i\in[N]}
\sum_{j\neq i}
\left(
L_{ij}^N+L_{ji}^N
\right).
$$
\end{proof}
In particular, a sufficient condition for $\alpha_N=O(N^{-1})$ is
\begin{align}
    \label{eq:suff_assumption}
    \max_{i\in[N]}
\sum_{j\neq i}
\left(
L_{ij}^N+L_{ji}^N
\right)
\leq
\frac{C}{N}
\end{align}
for a constant $C$ independent of $N$. 
\begin{assumption}[Sufficient assumption] Condition~\ref{eq:suff_assumption} holds in either of the
following situations:
    \begin{enumerate}
    \item Each pairwise influence satisfies
    $$
    L_{ij}^N+L_{ji}^N
    \leq
    \frac{C}{N^2},
    \qquad
    i\neq j.
    $$
    This is a dense game with interactions of order $O(N^{-2})$.

    \item There exist sets of interacting neighbors
    $\mathcal N_i^N\subset[N]\setminus\{i\}$ such that
    $$
    \sup_{N\geq2}
    \max_{i\in[N]}
    |\mathcal N_i^N|
    \leq D,
    $$
    $L_{ij}^N=L_{ji}^N=0$ whenever
    $j\notin\mathcal N_i^N$, and
    $
    L_{ij}^N+L_{ji}^N
    \leq
    \frac{C}{N} $
    on the interaction graph. This is a sparse game with bounded
    degree.
\end{enumerate}
\end{assumption}

Next, we provide an example of $\alpha_N=O(N^{-1})$-potential  game.

\begin{example}
\label{example-1}    

Let
$
X_t=x_0+\sigma W_t,
t\ge0,
$
where $W$ is a standard Brownian motion. Fix constants
$
\rho>0,
K>0,
\kappa_1,\kappa_2\geq0,
$
with
$
\kappa_1+\kappa_2>0.
$
For a pure stopping profile
$\tau=(\tau^1,\ldots,\tau^N)\in{\mathcal T}^N$, define
$$
m_{-i}^{\tau}(t)
:=
\frac{1}{N-1}
\sum_{j\neq i}
\mathbf 1_{\{\tau^j\leq t\}}.
$$
The reward of player $i$ is
\begin{align}
\label{eq:weak_interaction_pure_game}
J_i^{N}(\tau)
:=
\mathbb E\left[
e^{-\rho\tau^i}
\left(
(X_{\tau^i}-K)^+
-
\frac{\kappa_1}{N}
m_{-i}^{\tau}(\tau^i)
-
\frac{\kappa_2}{N}
\left(
m_{-i}^{\tau}(\tau^i)
\right)^2
\right)
\right].
\end{align}
Notice that the payoff~\eqref{eq:weak_interaction_pure_game} is of the local
stopped-status form~\eqref{eq:stopped_mass_game}. More precisely, for
$a^{-i}=(a^j)_{j\neq i}\in\{0,1\}^{N-1}$, define
$
\overline a^{-i}
:=
\frac{1}{N-1}
\sum_{j\neq i}a^j
$
and
\begin{align}
\label{eq:weak_interaction_h}
h_i^N(x,t,a^{-i})
:=
e^{-\rho t}
\left[
(x-K)^+
-
\frac{\kappa_1}{N}\overline a^{-i}
-
\frac{\kappa_2}{N}
\left(
\overline a^{-i}
\right)^2
\right].
\end{align}
Then
$$
J_i^{N}(\tau)
=
\mathbb E\left[
h_i^N\left(
X_{\tau^i},
\tau^i,
\left(
\mathbf 1_{\{\tau^j\leq\tau^i\}}
\right)_{j\neq i}
\right)
\right].
$$
Next we consider the randomized payoff function defined in Section~\ref{sec: randomize_game}. For $q^{-i}=(q^j)_{j\neq i}\in[0,1]^{N-1}$, write
$$
m_{-i}(q)
:=
\frac{1}{N-1}
\sum_{j\neq i}q^j,\qquad
v_{-i}^N(q)
:=
\frac{1}{(N-1)^2}
\sum_{j\neq i}q^j(1-q^j).
$$
Let $(Y^j)_{j\neq i}$ denote independent Bernoulli random
variables with parameters $(q^j)_{j\neq i}$ and define
$$
M_{-i}^N
:=
\frac{1}{N-1}
\sum_{j\neq i}Y^j,
$$
then
$$
\mathbb E[M_{-i}^N]
=
m_{-i}(q),\qquad
\mathbb E\left[
\left(M_{-i}^N\right)^2
\right]
=
m_{-i}(q)^2+v_{-i}^N(q).
$$
Consequently, Proposition~\ref{prop:canonical_local_representation}
gives the canonical independently randomized payoff
\begin{align}
\label{eq:weak_interaction_randomized_game}
J_i^{R,N}(p)
=
\mathbb E\left[
\int_{[0,\infty)}
g_i^N(X_t,t,p_t^{-i})\,dp_t^i
\right],
\tag{SMG'}
\end{align}
where
\begin{align}
\label{eq:weak_interaction_local_kernel}
g_i^N(x,t,q^{-i})
:=
e^{-\rho t}
\left[
(x-K)^+
-
\frac{\kappa_1}{N}m_{-i}(q)
-
\frac{\kappa_2}{N}
\left(
m_{-i}(q)^2+v_{-i}^N(q)
\right)
\right].
\end{align}
Fix a base point
$
z=(z^1,\ldots,z^N)\in\mathcal R^N.
$
The associated potential functional is
\begin{align}
\label{eq:weak_interaction_line_integral}
\Phi_{N,z}^R(p)
=
\int_0^1
\sum_{i=1}^N
\mathbb E\Bigg[
\int_{[0,\infty)}
g_i^N\left(
X_t,t,
\left\{
z_t+r(p_t-z_t)
\right\}^{-i}
\right)
d(p_t^i-z_t^i)
\Bigg]dr.
\end{align}
Moreover, one can show that 
the functional $\Phi_{N,z}^R$ defined
in~\eqref{eq:weak_interaction_line_integral} is an
$\alpha_N$-potential function for the randomized stopping
game~\eqref{eq:weak_interaction_randomized_game}, where
$
\alpha_N
\leq
\frac{4}{N}
\left(
\kappa_1+
\frac{2N-3}{N-1}\kappa_2
\right)
\leq
\frac{4(\kappa_1+2\kappa_2)}{N}.
$
In particular,
$
\alpha_N=O(N^{-1}).
$

To see this, 
fix distinct players $i,j\in[N]$. For
$a^{-ij}=(a^k)_{k\notin\{i,j\}}\in\{0,1\}^{N-2}$, define
$
S^{-ij}(a)
:=
\sum_{k\notin\{i,j\}}a^k$. Let
$a^{-ij}\oplus_j b$
denote the opponent-status vector of player $i$ obtained by assigning
the status $b\in\{0,1\}$ to player $j$ and retaining the coordinates
$a^{-ij}$ for all players other than $i$ and $j$.
The discrete influence of player $j$ on player $i$ is
$$
\begin{aligned}
\Delta_jh_i^N(x,t,a^{-ij})
:={}&
h_i^N\left(
x,t,a^{-ij}\oplus_j1
\right)
-
h_i^N\left(
x,t,a^{-ij}\oplus_j0
\right)=
-\frac{e^{-\rho t}}{N}
\left[
\frac{\kappa_1}{N-1}
+
\frac{\kappa_2}{(N-1)^2}
\left(
2S^{-ij}(a)+1
\right)
\right].
\end{aligned}
$$
Since
$
0\leq S^{-ij}(a)\leq N-2$,
we obtain
$
L_{ij}^N
\leq
\frac{1}{N}
\left[
\frac{\kappa_1}{N-1}
+
\frac{(2N-3)\kappa_2}{(N-1)^2}
\right].
$
Since the game is homogeneous, and hence the same bound holds for
$L_{ji}^N$. Proposition~\ref{prop:local_influence_bound} therefore gives
$$
\begin{aligned}
\alpha_N
&\leq
2
\max_{i\in[N]}
\sum_{j\neq i}
\left(
L_{ij}^N+L_{ji}^N
\right)\leq
4(N-1)
\frac{1}{N}
\left[
\frac{\kappa_1}{N-1}
+
\frac{(2N-3)\kappa_2}{(N-1)^2}
\right]
\\
&=
\frac{4}{N}
\left(
\kappa_1+
\frac{2N-3}{N-1}\kappa_2
\right)
\leq
\frac{4(\kappa_1+2\kappa_2)}{N},
\end{aligned}
$$
Hence the $O(N^{-1})$ estimate.

\end{example}

\subsection{Analysis of the stopping game with stopped-mass interaction}
\label{sec:sing_control}

To derive the $\alpha$-NE for~\eqref{eq:stopped_mass_game}, given the construction of the potential function in Proposition~\ref{prop:canonical_local_representation}, we now focus on the following optimization problem:
$$ V_0 := \sup_{\bm{p} \in \cR^N} \phi(\bm{p}). $$
One crucial observation is that this can be seen as a multidimensional singular control problem. Indeed, for each $p \in \cR_T$, there exists a nondecreasing adapted process $\n$ taking its values in $[0,+\infty)$ such that:
$$ p_t = 1 - e^{-\n_t} \ \mbox{for all} \ t \in [0,\infty). $$
From~\eqref{eq:local_potential}, we may then formulate a dynamic version of the potential problem, which is:
\bea\label{dynamic-potential-problem}
V(t, x, \bm{p}) := \sup_{\bm{\n} = (\n^1, \dots, \n^N)} \sum_{i=1}^N \dbE\Big[\int_t^\infty G_i\Big(X_s, s, P_{s-}^{-i}\Big)dP_s^i \Big| X_t = x, \bm{P}_{t-} = \bm{p} \Big],
\eea
with $\bm{P} := (P_1, \dots, P_N)$, each $P^i_\cdot$ driven by the dynamics:
\begin{equation}\label{dynamics-P} \left\{\begin{array}{ll}
dP_s^i = (1-P_s^i)d\n_s^i \q \mbox{whenever} \ \n^i \ \mbox{is continuous in} \ s, \\
\Delta P_s^i = (1-p_i)(e^{-\n_{s-}^i} - e^{-\n_s^i}) \q \mbox{whenever} \ \Delta \n_s^i \neq 0, 
\end{array}\right.
\end{equation}
and 
$$ G(x, t, p) := \int_0^1 g(x, t, rp)dr,$$
where we chose the basepoint according to Remark \ref{rem:base_point_zero}. By Remark 3.3, $z_t^\infty=0$ at every finite time, while the signed base-point mass is concentrated at $\infty$, where the payoff is defined to be zero. Hence the line-integral potential reduces exactly to the displayed singular-control objective.

 In this section, we assume that $X$ satisfies the stochastic differential equation (SDE):
\begin{equation}\label{dynamics-X} 
dX_s = b(X_s)ds + \sigma(X_s)dW_s, 
\end{equation}
where $b$ and $\sigma$ are bounded, uniformly Lipschitz-continuous in their second variable, and $\sigma$ satisfies \ $\si \ge \si_0 > 0$ for some constant $\si_0$.

 We also drop the dependence of $G$ in the variable $t$, which leads to the homogeneous control problem:
\bea\label{dynamic-potential-problem-infinite}
V(x, \bm{p}) := \sup_{\bm{\n} = (\n^1, \dots, \n^N)} \sum_{i=1}^N \dbE\Big[\int_0^\infty e^{-\rho t}G_i\Big(X_t, P_{t-}^{-i}\Big)dP_t^i \Big| X_0 = x, \bm{P}_{0-} = \bm{p} \Big],
\eea
which we will prove to be characterized by the following system:
\bea\label{DPE-infinite-horizon}
\left\{\begin{array}{ll}
\min\Big(\rho u -\cL^x  u, \ \min_{i \in [N] : p_i < 1} \{ -\pa_{p_i} u - G_i \}  \Big) = 0, \ u|_{\bm{p} = \bm{1}} = 0, \quad \mbox{whenever} \ \cM u(x,\bp,\bp') > 0 \ \mbox{for all} \ \bp' > \bp, \\
\cM u(x,\bp,\bp') \ge 0 \ \mbox{for all} \ \bp' \ge \bp,
\end{array}\right.
\eea
where $\cL^x u := b \pa_x u + \frac{1}{2}\si^2 \pa_{xx}^2 u$, and 
$$\cM u(x, \bp, \bp') :=- (u(x,\bm{p'}) - u(x,\bm{p})) - \cG(x, \bm{p}) \cdot (\bm{p'} - \bm{p}),$$ $\cG(b,\bm{p}) := \big(G_1(x,\bm{p}), \dots, G_N(x, \bm{p}))$ (trivially extending the functions $(G_1, \dots, G_N)$ from functions defined on $\dbR \times [0,1]^{N-1}$ to functions defined on $\dbR \times [0,1]^N$), and the inequality $\bm{p'} \ge \bm{p}$ is understood coordinatewise.

\begin{remark}
    Since the reward functions \((G_i)_{i\in[N]}\) depend on the controlled state \(\bP\) and are integrated against the corresponding components \(P^i\), the local gradient constraints in \eqref{DPE-infinite-horizon} only control infinitesimal, and hence continuous, increases of the stopping masses. In general, they do not imply the monotonicity inequality required for finite jumps of \(\bP\). This is why the dynamic programming equation must be supplemented with the nonlocal constraint encoded by the operator \(\cM\). This local–nonlocal separation between continuous stopping and jumps is reminiscent of the dynamic programming structure arising  in mean-field optimal stopping \cite{TalbiTouziZhang2023, TalbiTouziZhang2023Viscosity}.
\end{remark}

\begin{definition}
Let $u : \dbR \times [0,1)^N \to \dbR$. \\
    \noindent {\rm (i)} We say $u$ is a viscosity supersolution of \eqref{DPE-infinite-horizon} if, for all $(x,\bp) \in \dbR \times [0,1)$, $\d>0$ and $\f \in C^{2,1}(\dbR \times [0,1)^N)$ s.t.\ $(\f - u)(x, \bp) = 0 = \max_{\cB_\d(x, \bp)} (\f-u)$, we have:
    $$ \min\Big(\rho \f -\cL^x  \f, \ \min_{i \in [N] : p_i < 1} \{ -\pa_{p_i} \f - G_i \}  \Big) \ge 0, $$
    and $\cM u(x, \bp, \bp') \ge 0$ for all $\bp' \ge \bp$. \\
    \noindent {\rm (ii)} We say $u$ is a viscosity subsolution of \eqref{DPE-infinite-horizon} if, for all $(x,\bp) \in \dbR \times [0,1)$, $\d>0$ and $\f \in C^{2,1}(\dbR \times [0,1)^N)$ s.t.\ $(\f - u)(x, \bp) = 0 = \min_{\cB_\d(x, \bp)} (\f-u)$, we have:
    $$ \min\Big(\rho \f -\cL^x  \f, \ \min_{i \in [N] : p_i < 1} \{ -\pa_{p_i} \f - G_i \}  \Big) \le 0$$
    whenever $\cM u(x, \bp, \bp') > 0$ for all $\bp' > \bp$. \\
    \noindent {\rm (iii)} We say $u$ is a viscosity solution of \eqref{DPE-infinite-horizon} if it is a viscosity supersolution and subsolution. 
\end{definition}
 
\begin{theorem}
\label{thm:viscosity}
Assume the functions $(G_i)_{i \in [N]}$ are uniformly bounded and Lipschitz-continuous in their two variables. Then $V$ is a continuous viscosity solution of \eqref{DPE-infinite-horizon}. 
\end{theorem}
\proof
\textit{Continuity of $V$}. Fix a control $\bm{\n} := (\n_1, \dots, \n_N)$, $x, \tilde x \in \dbR$, $\bm{p} := (p_1, \dots, p_N)$ and $\tilde{\bm{p}} = (\tilde p_1, \dots, \tilde p_N)$ in $[0,1]^N$. We respectively denote by $X$ and $\tilde X$ the dynamics \eqref{dynamics-X} starting from $x$ and $\tilde x$ at $t=0$, and by $\bp$ and $\tilde \bp$ the dynamics \eqref{dynamics-P} starting from $\bp$ and $\tilde \bp$.

Observe that:
\begin{align*}
\Big| G_i\Big(X_t, P_t^{-i, -k} \otimes_k \tilde P_t^k \Big)dP_t^i - G_i\Big(X_t, P_t^{-i}\Big)d P_t^i \Big| \le& \frac{L}{N-1} | \tilde P_t^k - P_t^k | dP_t^i \\
=& L | (\tilde p_k - p_k)(1-e^{-\n_t^k}) | dP_t^i \\
\le& L |\tilde p_k - p_k| dP_t^i,
\end{align*}
and
\begin{align*}
\Big| G_i\Big(X_t, P_t^{-i} \Big)dP_t^i - G_i\Big(X_t, P_t^{-i}\Big)d\tilde P_t^i \Big| \le& | G |_\infty | dP_t^i - d\tilde P_t^i | 
\le | G |_\infty| p_i - \tilde p_i | d(e^{-\n^i_t})
\end{align*}
Then, denoting by $\bm{e}_i$ the element of $\dbR^N$ with $i$th coordinate equal to $1$ and all others equal to $0$, we have:
\begin{align*}
    | V(x, \bm{p}) - V(x, \bm{p} + (\tilde p_k - p_k)\bm{e}_k) | \le& \sup_{\bm{\n}} \Big\{| p_i - \tilde p_i |\dbE\Big[ L\sum_{k \neq i} \int_0^\infty e^{\rho t} dP_t^k + | G |_\infty \int_0^\infty e^{-\rho t} d(e^{-\n_t^i}) \Big]\Big\} \\
    \le& \big(L + | G |_\infty \big)| p_i - \tilde p_i |.
\end{align*}
Thus, $V$ is Lipschitz-continuous with respect to $\bm{p}$. For the continuity with respect to $b$, we simply observe that:
\begin{align*}
\dbE\Big[\Big| G_k\Big(X_t, P_t^{-k} \Big)dP_t^k - G_k\Big(\tilde X_t, P_t^{-k}\Big)d P_t^k \Big|\Big] \le& L\dbE\Big[| X_t - \tilde X_t | dP_t^k \Big] 
\le L | x - \tilde x|, \\
\end{align*}
from which we immediately obtain the Lipschitz-continuity of $V$ in $x$. \\

\textit{Viscosity property.} Since $V$ is continuous, it satisfies the dynamic programming principle:
\bea\label{DPP}
V(x, \bm{p}) = \sup_{\boldsymbol{\nu}}  \dbE\Big[\sum_{i=1}^N \,\int_0^\th e^{-\rho t} G\Big(X_t, P_{t-}^{-i} \Big) dP_t^i + e^{- \rho \th} V(X_\th, \bm{P}_\th) \Big].
\eea
We observe that playing a control $\n$ jumping at $0$ and remaining constant on $(0, \infty)$ immediately provides the inequality in \eqref{DPE-infinite-horizon}.
We first prove the viscosity supersolution property. Let $\f \in C^{2,1}(\dbR \times [0,1)^N)$ be such that $(\f - V)(x, \bm{p}) = 0 = \max_{\cB_\d(b, \bm{p})}(\f - V)$ for some $\d > 0$. The dynamic programming principle \eqref{DPP} provides:
$$ V(x, \bm{p}) \ge  \dbE\Big[\sum_{i=1}^N \,\int_0^\th e^{-\rho t} G\Big(X_t, P_{t-}^{-i} \Big) dP_t^i + e^{- \rho \th} \f(X_\th, \bm{P}_\th) \Big] $$
for all control $\n$. Considering first a constant $\n$, we obtain after applying Itô's formula along with the usual localization arguments:
$$ \rho \f(x, \bm{p}) - \cL^x \f(x, \bm{p}) \ge 0. $$
We now consider a constant control $\boldsymbol{\n}^i$ defined coordinatewise by $d\n_t^j = \e^{-1}\1_{j = i}dt$. Applying again Itô's formula and localization arguments, we obtain:
$$ \rho \f(x, \bm{p}) - \cL^x \f(x, \bm{p}) + \e^{-1}\big(- \pa_{p_i}\f(x,\bm{p}) - G_i(x,\bm{p})\big)  \ge 0, $$
and we obtain after multiplying by $\e$ and letting $\e \to 0$:
$ - \pa_{p_i}\f(x,\bm{p}) - G_i(x,\bm{p}) \ge 0, $
which provides the supersolution property. \\

We now prove the subsolution property. Let $\bp \in [0,1]^N \setminus \{(1, \dots, 1)\}$ s.t.\ $\cM u(x,\bp, \bp') > 0$ for all $\bp' > \bp$. 
Let $\f \in C^{2,1}(\dbR \times [0,1))$ be such that $(\f - V)(x, \bm{p}) = 0 = \min_{\cB_\d(b, \bm{p})}(\f - V)$ for some $\d > 0$, and:
\bea\label{ineq-viscosity} 
\min\Big(\rho \f -\cL^{x} \f, \ \min_{i \in [N] : p_i < 1} \{ -\pa_{p_i} \f - G_i \}  \Big) > 0.  
\eea
By continuity of $\f$ and $G_i$, the above inequality remains true on some ball $\cB_{\d'}(x, \bm{p})$, with $\d' \le \d$, and we have in particular for all $(x',\bm{p}')$ and $(x',\bm{p}'') \in \cB_{\d'}(b,\bm{p})$:
\begin{align*}
    \f(x', \bm{p}') - \f(x', \bm{p}'') &\ge \int_0^1 \Big(\cG(x', \bm{p}' + \lambda(\bm{p}'' - \bm{p}') + \gamma \Big)\cdot (\bm{p}' - \bm{p}'') \\
    &\ge \Big(\cG(x', \bm{p}') - L\d' + \gamma)\cdot (\bm{p}' - \bm{p}'').
\end{align*}
by Lipschitz-continuity of $\cG$ in its second argument. By choosing $\d'$ sufficiently small, we then have:
\bea\label{jump-test-function}
\f(x', \bm{p}') - \f(x', \bm{p}'') \ge \cG(x', \bm{p}')(\bm{p}' - \bm{p}''). 
\eea
Note also that we may assume without loss of generality that $\f$ takes the form:
$$ \f(x', \bm{p}') = V(x, \bm{p}) + \a | x' - x |^4 + \b | \bm{p}' - \bm{p} |^2 \quad \mbox{for some} \ \a, \b > 0. $$
We finally denote by $\th_{\d'}$ the first time the process $(X_t, \bm{P}_t)$ (starting from $(X_0, \bP_{0-}) = (x, \bm{p})$) exits the ball $\cB_{\d'}(x, \bm{p})$. By property of the Brownian motion, we have $\th_{\d} < \infty$, a.s. We now fix an arbitrary control $\boldsymbol{\n}$. We distinguish between two cases: \\
{\textit Case 1:} $(x, \bm{P}_0) \notin \cB_{\d'}(x, \bm{p})$. Then we have:
\bea\label{nonDPP1}
V(x,\bm{p}) \ge \eta + V(x, \bm{P}_0) + \cG(x, \bm{P}_0) \cdot (\bm{P}_0 - \bm{p}), 
\eea
with $\eta := \sup_{\bm{p}' \in [\bm{p}, \bm{p} + r]} \cM V(x, \bm{p}, \bm{p}') > 0$ as $\cM V(x, \bm{p}, \bm{p}') > 0$ for all $\bm{p}' > \bm{p}$. \\

{\textit Case 2:} $(x, \bP_0) \in \cB_{\d'}(x, \bp)$. Then, by right continuity of $s \mapsto \bP_s$, we have $\th_{\d'} > 0$, a.s. We then apply Itô's formula to $\f(X_t, \bm{P}_t)$ between times $0$ and $\th_{\d'}$:
\begin{align*}
\f(x, \bm{p}) =& \dbE\Big[ e^{-\rho \th_{\d'}} \f(X_{\th_{\d'}}, \bm{P}_{\th_{\d'}}) + \int_0^{\th_{\d'}} e^{-\rho t}\big(\rho \f(X_t, \bm{P}_t) - \cL^x \f(X_t, \bm{P}_t) \big)dt + \sum_{i=1}^N -e^{-\rho t} \pa_{p_i} \f(X_t, \bm{P}_{t-}) dP_t^i \\
&- \sum_{0 \le s \le \th_{\d'}} e^{-\rho t}\Big( \f(X_t, \bm{P}_t) - \f(X_t, \bm{P}_{t-}) - \sum_{i=1}^N \pa_{p_i} \f(X_t, \bm{P}_{t-})\D P_t^i \Big)\Big] \\
\ge& \dbE\Big[ e^{-\rho \th_{\d'}} \f(X_{\th_{\d'}}, \bm{P}_{\th_{\d'}}) + \int_0^{\th_{\d'}} e^{-\rho t}\sum_{i=1}^N G_i(X_t, \bm{P}_{t-})dP_t^i \Big],
 \end{align*}
 where we used \eqref{jump-test-function}. Then, noticing that
 $$ \f(X_{\th_{\d'}}, \bm{P}_{\th_{\d'}}) \ge V(X_{\th_{\d'}}, \bm{P}_{\th_{\d'}}) + \a (\d')^4 \vee \b (\d')^2, $$
 we have:
 \bea\label{nonDPP2}
 \f(x, \bm{p}) = V(x, \bm{p}) \ge \a (\d')^4 \vee \b (\d')^2 + \dbE\Big[ e^{-\rho \th_{\d'}} V(X_{\th_{\d'}}, \bm{P}_{\th_{\d'}}) + \int_0^{\th_{\d'}} e^{-\rho t}\sum_{i=1}^N G_i(X_t, \bm{P}_{t-})dP_t^i \Big]. 
 \eea
 Finally, putting together \eqref{nonDPP1} and \eqref{nonDPP2}, we obtain contradiction of the dynamic programming principle \eqref{DPP}. Therefore \eqref{ineq-viscosity} is false and $V$ is a viscosity subsolution of \eqref{DPE-infinite-horizon}.
\qed 

\begin{remark}\label{rem:supersolution}
    While the equality in \eqref{DPE-infinite-horizon} only holds for $\bm{p}$ s.t. $\cM V(x, \bm{p}, \bm{p}')$ for all $\bm{p}' > \bm{p}$, we proved in fact that $V$ is a viscosity supersolution of
    $$ \min\Big( \rho u - \cL^x u, \min_{i \in [N] : p_i < 1} \{ - \pa_{p_i} u - G_i \} \Big) = 0$$
    in every $(x,\bp) \in \dbR \times [0,1]^N$. 
\end{remark}

\begin{theorem}
\label{thm:regularity-infinite-horizon}
Assume that $G$ is bounded, $C_b^2$ in the first variable and Lipschitz continuous in the second variable. Then $V$ defined in \eqref{dynamic-potential-problem-infinite} belongs to $W^{(2,1), \infty}$, that is, $V$ admits a first order weak derivative with respect to \ $p$ and a second order weak derivative w.r.t.\ $b$, and both are bounded.
\end{theorem}
\proof
We showed in the proof of Theorem \ref{thm:viscosity} that $V$ is Lipschitz-continuous with respect to each $p_i$, $i \in [N]$.
Thus, $V$ admits a bounded weak derivative w.r.t\ $p_i$ for every $i \in [N]$.

We then focus on the second order derivative in $b$. Let $\e \in \dbR$. We want to show that:
\bea\label{finite-difference}
-C\e^2 \le D_{x,\e}^2 V(x, \bm{p}) \le C\e^2 
\eea
for some constant $C \ge 0$ independent of $\e$, and
$$ D_{x,\e}^2 V(x, \bm{p}) :=  V(x + \e, \bm{p}) + V(x - \e, \bm{p}) - 2V(x, \bm{p}) .$$
We first focus on the lower bound. Introduce:
$$ J(b, \bm{p}, \bm{\n}) := \sum_{i=1}^N \dbE\Big[\int_0^\infty e^{-\rho t} G_i\Big(X_t, \bP_{t-}\Big)dP_t^i \Big| X_0 = b, \bm{P}_{0-} = \bm{p} \Big]. $$
By definition, we have:
\begin{align*}
 V(x + \e, \bm{p}) + V(x - \e, \bm{p}) - 2V(x, \bm{p}) \ge& J(x + \e, \bm{p}; \bm{\n'}) + J(x - \e, \bm{p}; \bm{\n'}) + 2 \inf_{\bm{\n}} - J(x, \bm{p}; \bm{\n}) \ge \inf_{\bm{\n}} \Big\{ D_{x,\e}^2  J(x, \bm{p}; \bm{\n} ) \Big\} \\
 =& \inf_{\bm{\n}} \sum_{i=1}^N \dbE\Big[\int_0^\infty e^{-\rho t}D_{x,\e}^2 G_i\Big(X_t, \bP_{t-}\Big)dP_t^i \Big| X_0 = x, \bm{P}_{0-} = \bm{p} \Big].
 \end{align*}
 Since $G$ is $C^2$ with bounded derivatives in $x$, we have $ D_{x,\e}^2 G \ge {\color{green}-}\lVert \pa_{xx}^2 G \rVert_\infty \e^2$, and therefore: 
 \begin{equation}\label{ineq-second-order}
 D_{x,\e}^2 V \ge - \lVert \pa_{xx}^2 G \rVert_\infty \e^2 \sum_{i=1}^N (1 - p_i)
 \ge - N\lVert \pa_{xx}^2 G \rVert_\infty \e^2.
 \end{equation}
 Since all the processes $P^i$ have a total variation bounded by $1$, we deduce the left inequality of \eqref{finite-difference}. \\

 For the upper bound, we rely on the PDE characterization of $V$. By Remark \ref{rem:supersolution}, $V$ is a viscosity supersolution to $\rho u - \cL^x u = 0$, which implies that $-\partial_{xx}^2 V \ge -\frac{2}{\sigma_0^2}\big(\rho|V|_\infty + |b \pa_x V |_\infty\big) $ in the viscosity sense. Thus, the auxiliary continuous function $\tilde{V}(x, \mathbf{p}) := V(x, \mathbf{p}) - \frac{C’}{2}b^2$, with $C' := \frac{2}{\sigma_0^2}\big(\rho|V|_\infty + |b \pa_x V |_\infty\big)$, is a viscosity supersolution to $-\partial_{xx}^2 \tilde{V} \ge 0$. This implies that $\tilde{V}$ is concave in $x$. Writing the concavity inequality $\tilde{V}(x+\varepsilon, \bp) + \tilde{V}(x-\varepsilon, \bp) - 2\tilde{V}(x, \bp) \le 0$ and expanding the quadratic terms yields exactly $D_{x,\varepsilon}^2 V(x, \mathbf{p}) \le C\varepsilon^2$
 
 By Evans \cite{evans2022partial}, the double inequality \eqref{finite-difference} provides the desired Sobolev regularity. 
\qed

\begin{theorem}\label{thm:verification}
    The value function $V$ \eqref{DPE-infinite-horizon} is the unique $W^{(2, 1), \infty}$ solution to the dynamic programming equation \eqref{DPE-infinite-horizon}.
\end{theorem}
\proof
Let $u$ be a $W^{2, 1, \infty}$ solution to \eqref{DPE-infinite-horizon}. \\
\textbf{Step 1.} We first show that $u \ge V$. Fix $(x, \bm{p}) \in \dbR \times [0,1]^N$ and a control $\bm{\n}$. By Itô-Tanaka-Meyer formula, we have:
\begin{align*}
    u(x, \bm{p}) =& \dbE\Big[ e^{-\rho t}u(X_t, \bm{P}_t) + \int_0^t e^{-\rho s} \Big((\rho u(X_s, \bm{P}_s) - \cL^x u(X_s, \bm{P}_s))ds - \sum_{i=1}^N \pa_{p_i} u(X_s, \bm{P}_{s-})dP_s^i\Big) \\ 
    &- \sum_{0 \le s \le t} e^{-\rho s}\Big( u(X_s, \bm{P}_s) - u(X_s, \bm{P}_{s-}) - \sum_{i=1}^N \pa_{p_i} u(X_s, \bm{P}_{s-})\D P_s^i \Big) | X_0 = x, \bm{P}_{0-} = \bm{p}  \Big],
\end{align*}
as the fact that $\pa_x u$ is bounded a.e.\ implies that the stochastic integral has a zero expectation. Now, using the fact that 
$$ - \Big( u(X_s, \bm{P}_s) - u(X_s, \bm{P}_{s-})\Big) \ge \cG(X_s, \bm{P}_{s-}) \cdot \D \bm{P}_s \quad \mbox{and} - \pa_{p_i}u \ge G_i, $$
plus the fact that $\rho u - \cL^x u \ge 0$, we deduce that, for all $t \ge 0$,
$$u(x, \bm{p}) \ge \dbE\Big[e^{-\rho t}u(X_t, \bm{P_t}) + \int_0^t e^{-\rho s}\cG(X_s, \bm{P}_{s-}) \cdot d \bm{P}_s | X_t = x, \bm{P}_{t-} = \bm{p} \Big]. $$
By boundedness of $u$, we obtain by letting $t \to \infty$:
$$ u(x, \bm{p}) \ge \dbE\Big[\int_0^\infty e^{-\rho s}\cG(B_s, \bm{P}_{s-}) \cdot d \bm{P}_s | X_t = x, \bm{P}_{t-} = \bm{p} \Big],$$
which implies by arbitrariness of the control that $u \ge V$. \\ 
\textbf{Step 2.} To prove the converse inequality, we construct a nearly optimal control.
Let $\varepsilon > 0$. We define, almost everywhere: $\boldsymbol{\lambda}^{\varepsilon} = (\lambda^{\varepsilon, 1}, \dots, \lambda^{\varepsilon, N})$ defined for each $i \in [N]$ by:
\begin{equation*}
    \lambda^{\varepsilon, i}(x, \mathbf{p}) = \e^{-1} \mathbf{1}_{\{ -\partial_{p_i}u(x, \mathbf{p}) - G_i(x, \mathbf{p}) \le \varepsilon \}},
\end{equation*}
and then set, for each $i \in [N]$:
\begin{equation*}
\n_t^{\varepsilon, i} :=  \lambda^{\varepsilon, i}(X_t, \mathbf{P}_t)\1_{\{ P_t^i < 1 - \sqrt{\e}\}} \quad \mbox{for all} \ t \ge 0. 
\end{equation*}
Because this control rate is bounded by $\e^{-1}$, the corresponding state dynamics admits a unique strong solution. Applying Itô-Tanaka-Meyer formula to $u$ along the dynamics driven by $\boldsymbol{\lambda}^{\varepsilon}$, and observing that the state process does not jump along this control, we obtain:
\begin{equation}\label{ito-proof1}
u(x, \bm{p}) = \dbE\Big[ e^{-\rho t} u(X_t, \bm{P}_t) + \int_0^t e^{-\rho s}\Big(\rho u(X_s, \bm{P}_s) - \cL^x u(X_s, \bm{P}_s)ds - \sum_{i=1}^N \pa_{p_i} u(X_s, \bm{P}_s)(1-P_s^i)\lambda_s^{\e, i}ds\Big)\Big]. 
\end{equation}
Now, observe that, by construction of $\boldsymbol{\lambda}^{\varepsilon}$:
$$ \int_0^t e^{-\rho s}\sum_{i=1}^N \pa_{p_i} u(X_s, \bm{P}_s)(1-P_s^i)\lambda_s^{\e, i}ds \le C\varepsilon + \int_0^t e^{-\rho s} \cG(X_s, \bm{P}_s) \cdot d\bm{P}_s, $$
for some constant $C \ge 0$. Also, by boundedness of $\pa_{xx}^2 u$ and \eqref{DPE-infinite-horizon}, we have:
$$ \int_0^t e^{-\rho s}\big[\rho u(X_s, \bm{P}_s) - \cL^x u(X_s, \bm{P}_s)\big]ds \le C\Big[ \varepsilon + {\rm{Leb}}\big(\{ s \in [0,t] : \lambda^{\e, i}_s = \e^{-1} \ \mbox{for all} \ i \in [N]\} \big)\Big]. $$
Introduce now the stopping time:
\bea\label{stopping-time}
  \t_\e^{(1)} := \inf\{ s \ge 0 : P_s^i \ge 1 - \sqrt{\e} \q \mbox{for some} \ i \in [N]\}. 
\eea 
We have:
\begin{align*}
    {\rm{Leb}}\big(\{ s \in [0,\t_\e^{(1)}] : \lambda^{\e, i}_s = \varepsilon^{-1} \ \mbox{for all} \ i \in [N]\}\big) &\le \e\int_0^{\t_\e^{(1)}} \sum_{i \in [N]} \l_s^{\e, i} ds = \e\int_0^{\t_\e^{(1)}} \sum_{i \in [N]} \frac{dp_s^i}{1 - p_s^i} \\
    &\le \frac{\e}{\sqrt{\e}}\int_0^{\t_\e^{(1)}ta} \sum_{i \in [N]} dp_s^i \le N\sqrt{\e} .
\end{align*}
Then we obtain from \eqref{ito-proof1}:
$$
u(x, \bm{p}) \le \dbE\Big[ e^{-\rho \t_\e^{(1)}} u(X_{\t_\e^{(1)}}, \bm{P}_{\t_\e^{(1)}}) + \int_0^{\t_\e^{(1)}} e^{-\rho s}\cG(B_s, \bm{P}_s) \cdot d\bm{P}_s \Big] + C_N\sqrt{\varepsilon},
$$
with $C_N$ a nonnegative constant depending on $N$. 
 Now, we set $P_t^i = 1$ on $[\t_\e^{(1)}, \infty)$ for all the indices $i$ satisfying the inequality in \eqref{stopping-time}.

Defining the (random) set $\cI_1 := \{ i \in [N] : P_{\t_\e}^i = 1 \}$, we may then introduce:
\bea\label{stopping-time2}
  \t_\e^{(2)} := \inf\{ s \ge \t_\e^{(1)} : P_s^i \ge 1 - \sqrt{\e} \q \mbox{for some} \ i \notin \cI_1\}. \nonumber 
\eea 
Then, we re-iterate the above reasoning to derive:
$$ u(x, \bm{p}) \le \dbE\Big[ e^{-\rho \t_\e^{(2)}} u(X_{\t_\e^{(2)}}, \bm{P}_{\t_\e^{(2)}}) + \int_0^{\t_\e^{(2)}} e^{-\rho s}\cG(X_s, \bm{P}_s) \cdot d\bm{P}_s \Big] + 2C_N \sqrt{\e}. 
$$
By induction, we re-iterate this construction with stopping times $\t_\e^{(k)}$ until all players are stopped (that is, until $\boldsymbol{P}_t = \mathbf{1}$) or until $\t_\e^{(k)} \ge \ln(\e)/(-\rho)$. In the first case, we use the boundary condition in \eqref{DPE-infinite-horizon} ; in the second case, we use the boundedness of $u$ and $\cG$. Both cases  lead to:

$$ u(x, \bm{p}) \le \dbE\Big[ \int_0^{\infty} e^{-\rho s}\cG(X_s, \bm{P}_s) \cdot d\bm{P}_s \Big] + NC_N \e \le V(x,\bm{p}) + NC_N \sqrt{\e}, $$
and therefore we conclude by arbitrariness of $\e > 0$ that $u = V$.
\qed

\section{Reinforcement learning algorithm for finite-player optimal stopping}
\label{sec: DDPG}
In this section,  we propose a reinforcement learning  algorithm for the randomized potential function~\eqref{eq:local_potential} when the model is unknown.  For later numerical algorithm, we consider the finite-time-truncated approximation of the pure optimal stopping game with finite time horizon $T>0$.
\subsection{Numerical algorithm}
Recall that  we first randomize the stopping time in Section~\ref{sec: randomize_game}, and approximate the finite-player optimal stopping game with an $\alpha_N$-potential function $\phi: (\cR_T)^N\to \dbR$ in~\eqref{eq:local_potential}. From Proposition ~\ref{prop:local_alpha_estimate}, every maximizer $\boldsymbol{p}^*\in\arg\max_{\boldsymbol{p}\in (\cR_T)^N} \phi(\boldsymbol{p})$ is an $\alpha_N$-Nash equilibrium of the randomized game~\eqref{eq:randomized_game}. For every $i\in [N]$ and each $p^i \in \cR_T$, there exists a non-decreasing adapted process $\n^i$ taking its values in $[0,+\infty)$ such that:
$$ p_t^i = 1 - e^{-\n_t^i} \ \mbox{for all} \ t \in [0,T]. $$

Since this is a singular control problem, we further regularize and approximate the process $\nu^i$ by an absolutely continuous intensity control
    $$
    d\nu_t^i = \lambda_t^i\,dt,
    \qquad
    \lambda_t^i \ge 0,
    $$ 
where for every $i\in [N]$,
$$\lambda^i\in \cA_T := \{ \lambda : [0,T] \times \O \to [0,\lambda_{\max}] \ \  \mbox{such that} \  \lambda  \ \mbox{is RCLL and adapted} \},$$
and the controlled vector is
$$
\lambda=(\lambda^1,\ldots,\lambda^N)\in (\cA_T)^N.
$$
Hence the controlled stopping fractions satisfy
$$
dp_t^i = (1-p_t^i)\lambda_t^i\,dt,
\qquad i=1,\dots,N,
$$
and the regularized potential control problem is then
\begin{equation}
\label{eq: intensity_control}
    \sup_{\lambda\in (\cA_T)^N} \Psi(\lambda):=\sup_{\lambda\in (\cA_T)^N} \phi(\boldsymbol{p}(\lambda)).
\end{equation}
When the model is unknown,  we  adopt the CT-DDPG algorithm proposed in~\cite{ChengGuoZhang2025} to solve the control problem~\eqref{eq: intensity_control}. The pseudo-code is illustrated in Algorithm~\ref{alg:general-finite-player-stopping}.

\label{sec:num_exp}
\subsection{Experiments with exact potential stopping game}
\label{sec: potential_model}
We first consider the exact potential game introduced in Section~\ref{sec:exact_potential_example}, and use its closed-form solution as the benchmark.

To apply Algorithm~\ref{alg:general-finite-player-stopping}, we first replace the hard stopping by a stopping-fraction state $p_t\in[0,1]$ and a nonnegative intensity control $\lambda_t\in [0,\lambda_{\max}]$. Since the game is symmetric and the exact equilibrium is symmetric, we will only consider using one common stopping fraction and one common intensity. The regularized dynamics are
$$
d X_t=\sigma\,dW_t,
\qquad
dp_t=(1-p_t)\lambda_t\,dt,
\qquad
\lambda_t\in[0,\lambda_{\max}],
$$
with initial condition
$$
X_0=x,
\qquad
p_0=0.
$$
Here $p_t$ is interpreted as the fraction of stopping mass already exercised by time $t$, while $1-p_t$ is the residual unstopped mass.

Fix a truncated time horizon $T>0$, the regularized objective is
$$
J(x,\lambda)
=
\dbE_{x}\!\left[
\int_0^T e^{-\rho t}N(X_t-K)^+(1-p_t)\lambda_t\,dt
\right].
$$
We parametrize the deterministic policy by $\lambda_\phi$, the value function by $V_\theta$, and the raw advantage-rate function by $\bar q_\psi$. The reparameterized advantage is
$$
q_\psi(s,\lambda)=\bar q_\psi(s,\lambda)-\bar q_\psi(s,\lambda_\phi(s)),
$$
so that
$$
q_\psi(s,\lambda_\phi(s))=0.
$$

In all experiments, we fix the game parameters as
$$
N=2,\qquad \eta=0.5,\qquad x_0=0,\qquad K=1,\qquad \sigma=1.5,\qquad \rho=2.
$$
The corresponding closed-form stopping threshold is
$
x^* = K+\frac{\sigma}{\sqrt{2\rho}} = 1.75.
$
For the numerical scheme, we use terminal horizon $T=20$,
time step
$
\Delta t = 0.05$, 
and hence $400$
time steps per episode.

The actor is parameterized by a threshold-type deterministic intensity policy $\lambda_\phi(X)
=
\lambda_{\max}\,\sigma\!\bigl(c(X-b)\bigr)$, where $b$ is the learned soft stopping boundary and $c$ is the learned slope. The critic is factorized as $V_\theta(X,p)=(1-p)\,U_\theta(X)$, and the raw advantage network is parameterized as $\bar q_\psi(X,p,a)=(1-p)\,\bar Q_\psi(X,a)$. Both $U_\theta$ and $\bar Q_\psi$ are implemented by fully connected neural networks with two hidden layers of width $64$, ReLU activation functions, and scalar output.

For training, we use an $L$-step temporal-difference target with $L=10$, replay buffer size $50000$,
batch size
$64$,
and soft target-network update parameter
$\tau=0.01$. The learning rates are
$$
\texttt{lr\_actor}=10^{-6},
\qquad
\texttt{lr\_critic}=10^{-2}.
$$
Exploration is implemented by Gaussian perturbation of the deterministic actor output. The exploration standard deviation is initialized at $0.3$,
decays geometrically with factor
$0.995$,
and is lower bounded by
$0.02$.
We train for $800$ episodes and perform $10$
gradient updates per episode.

\medskip
The learned actor and critic are plotted in Figure~\ref{fig:potential_intensity}, where the left figure illustrates the learned intensity, and the right figure shows the learned value function (as a function of $x$). We can see that the learned intensity curve   approximates closely the optimal stopping time.

\begin{figure}
    \centering
    \includegraphics[width=0.4\linewidth]{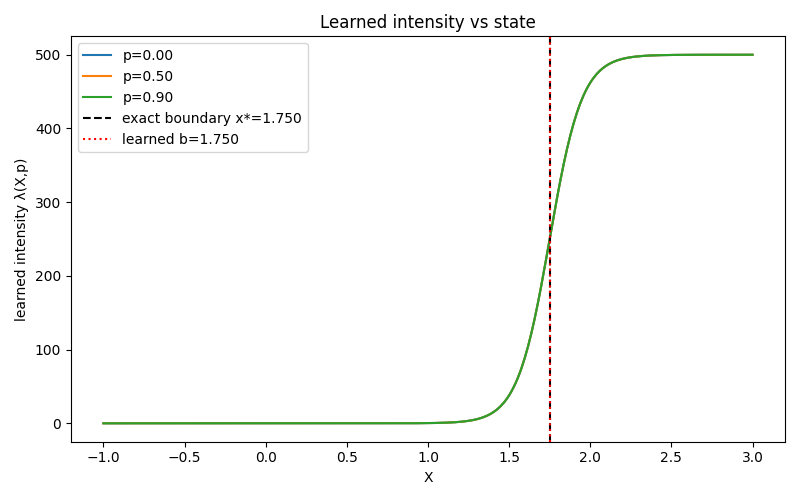}
    \includegraphics[width=0.4\linewidth]{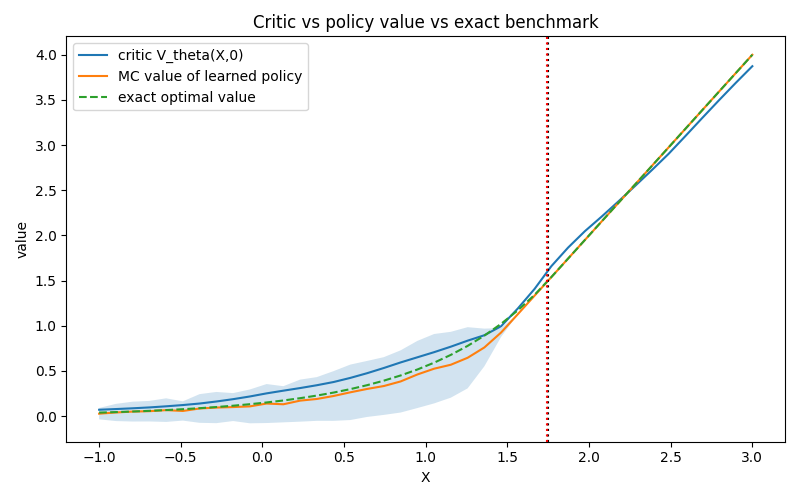}
    \caption{Left: Learned intensity.   Right: Learned value function.}
    \label{fig:potential_intensity}
\end{figure}

\subsection{Experiment with $\alpha$-potential stopping game}\label{sec:experiment-alpha}
We next test the proposed method for the finite-player optimal stopping game introduced in Section~\ref{sec:separation_game}.
In all experiments, we set
$
x_0=0, K=1, \sigma=1.5, \rho=2,
\kappa_1=0.5, \kappa_2=2.0,
$
and consider
$
N\in\{2,5,10,20, 40\}.
$

For numerical implementation, we adopt a bounded absolutely continuous intensity control
$$
\mathrm d\nu_t=\lambda_t\mathrm dt,
\qquad
\mathrm dp_t=(1-p_t)\lambda_t\,\mathrm dt,
$$
and discretize the state dynamics by
$$
X_{k+1}=X_k+\sigma\sqrt{\Delta t}\,\xi_k,
\qquad
\xi_k\sim\mathcal N(0,1).
$$
We use
$
T=20, \Delta t=0.05, \gamma=e^{-\rho\Delta t},
$
so that each episode contains $400$ time steps. For numerical stability, the maximal intensity $\lambda_{\max}$ is set to be $200$.

The critic consists of a value network $V_\theta$ and a raw advantage-rate network $\bar q_\psi$, each implemented as a fully connected neural network with two hidden layers of width $64$ and ReLU activations. The advantage-rate is reparameterized as
$$
q_\psi(s,a)=\bar q_\psi(s,a)-\bar q_\psi\bigl(s,\mu_\phi(s)\bigr),
$$
as in CT-DDPG. We use a replay buffer of size $50{,}000$, batch size $64$, multi-step TD length $L=10$, soft target update parameter $\tau=0.005$,
critic learning rate
$5\times 10^{-4}$, actor learning rate
$10^{-4}$, and terminal penalty weight equal to $1$. Exploration is Gaussian around the deterministic actor output with initial standard deviation $0.15$, geometric decay factor $0.995$, and lower bound $0.02$. In each episode we perform $10$ gradient updates. The centralized potential phase is trained for $700$ episodes, and the best-response phase for $500$ episodes. Actor updates are delayed: in the potential phase the actor is frozen for the first $150$ episodes, and in the best-response phase for the first $100$ episodes; afterwards the actor is updated once every $5$ critic updates. All runs start from $x_0=0$, and evaluation is performed every $25$ episodes. 

For every $N$, denote $\hat p^N$ the learned symmetric optimal policy from the $\alpha_N$-potential function $\Phi_N^{\mathrm{sym}}(p)$. Let $\hat p^{N,\mathrm{BR}}$ denote the numerically computed best response against $\hat p^N$.

\begin{figure}
    \centering
    \includegraphics[width=0.3\linewidth]{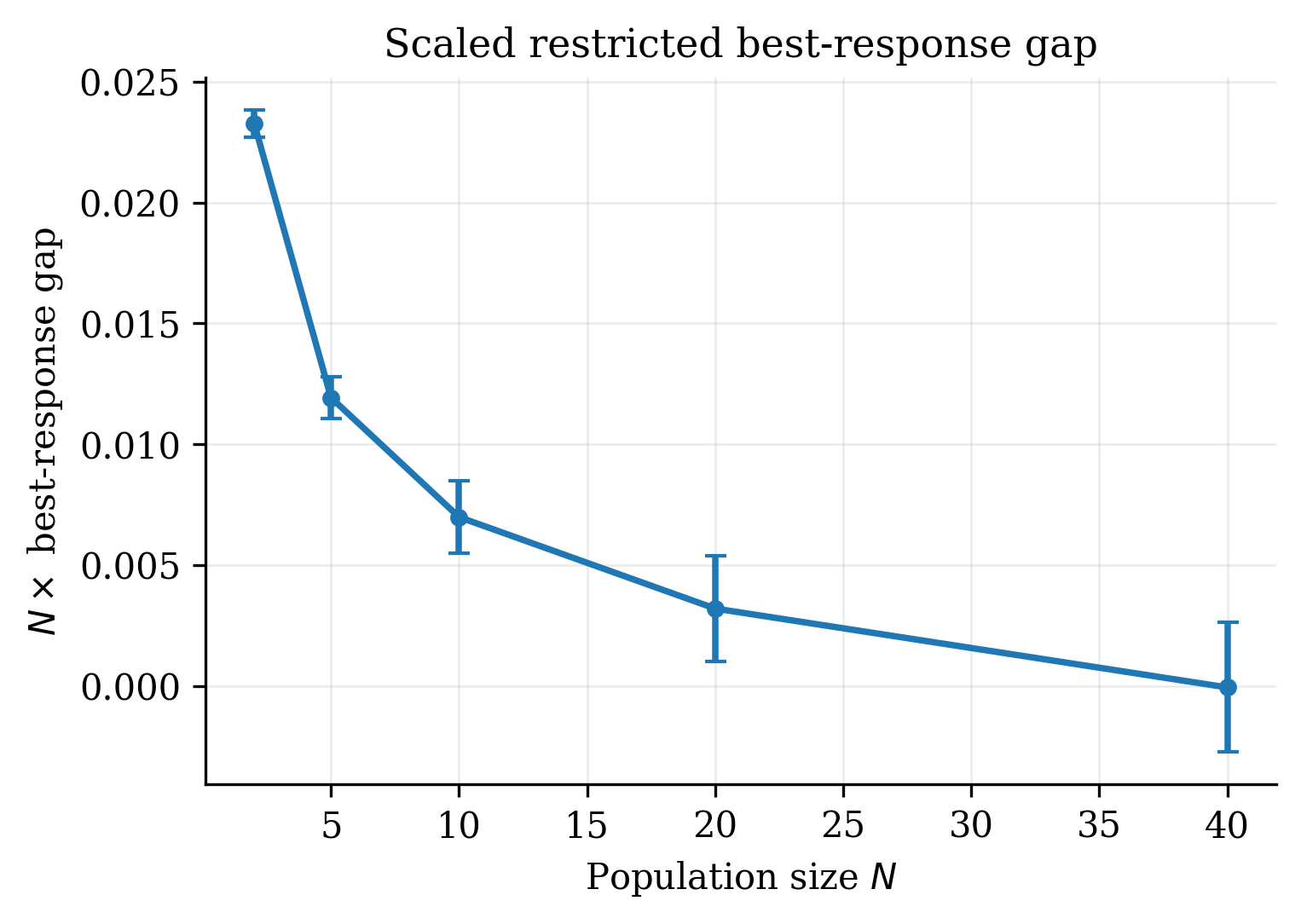}
    \caption{Gap with number of agents $N$.}
    \label{fig:gap_num}
\end{figure}

\begin{algorithm}[H]
\label{alg:general-finite-player-stopping}
\caption{Potential-CT-DDPG}
\begin{algorithmic}[1]
\Require
number of players $N$, time step $h$, horizon $H$, discount $\gamma=e^{-\rho h}$, multi-step TD length $L$, replay buffer $\mathfrak R$, batch size $B$, actor $\mu_\phi$, value net $V_\theta$, raw advantage net $\bar q_\psi$, target value net $V_{\theta^-}$, exploration scale $\sigma_{\mathrm{exp}}$, critic learning rates $\eta_\theta,\eta_\psi$, actor learning rate $\eta_\phi$, terminal  weight $\alpha_T$, soft-update parameter $\tau$.
\State Initialize $(\phi,\theta,\psi)$ and set $\theta^- \gets \theta$
\For{episode $n=1,\dots,N_{\mathrm{epi}}$}
    \State Initialize exogenous state $X_0$ and stopping fractions $p_0=(0,\dots,0)$
    \For{$k=0,\dots,H-1$}
        \State Observe state
        $
        s_k=(t_k,X_k,p_k)
        $, simulate noise $\xi_k$
        \State Compute deterministic intensity
        $
        \bar \lambda_k=\mu_\phi(s_k)\in [0,\lambda_{\max}]^N
        $
        \State Add local exploration
        $
        \lambda_k
        =
        \Pi_{[0,\lambda_{\max}]^N}\bigl(\bar \lambda_k+\sigma_{\mathrm{exp}}\zeta_k\bigr),
        \qquad
        \zeta_k\sim N(0,I_N)
        $
        \State Simulate the exogenous state
        $
        X_{k+1}=\mathsf{SimState}(t_k,X_k,\lambda_k,\xi_k)
        $
        \State Update randomized stopping fractions
        $$
        \Delta p_k^i=(1-p_k^i)\bigl(1-e^{-h \lambda_k^i}\bigr),
        \qquad
        p_{k+1}^i=p_k^i+\Delta p_k^i,
        \qquad i=1,\dots,N
        $$
        \State Collect stage reward
        $
        r_k
        $
        \State Store transition $(s_k,\lambda_k,r_k,s_{k+1})$ in replay buffer $\mathfrak R$
    \EndFor

    \For{each gradient step}
        \State Sample $B$ trajectory fragments of length $L$ from $\mathfrak R$
        \State Define the reparameterized advantage
        $
        q_\psi(s,\lambda)
        =
        \bar q_\psi(s,\lambda)-\bar q_\psi(s,\mu_\phi(s))
        $
        \State For each sampled fragment, form the multi-step target
        $$
        Y^{(b)}
        =
        \sum_{\ell=0}^{L-1}
        \gamma^\ell
        \Bigl(
        r_{k+\ell}^{(b)}
        -
        h\,q_\psi\bigl(s_{k+\ell}^{(b)},\lambda_{k+\ell}^{(b)}\bigr)
        \Bigr)
        +
        \gamma^L
        V_{\theta^-}\bigl(s_{k+L}^{(b)}\bigr),
        \qquad b=1,\dots,B
        $$
        \State Sample $B$ terminal states/payoffs $(s_K^{(b)},R_T^{(b)})$ from $\mathfrak R$
        \State Compute critic loss
        $$
        \mathcal L_{\mathrm{M}}
        =
        \frac1B\sum_{b=1}^B
        \Bigl(
        V_\theta\bigl(s_k^{(b)}\bigr)-Y^{(b)}
        \Bigr)^2,\quad
        \mathcal L_{\mathrm{C}}
        =
        \frac1B\sum_{b=1}^B       \Bigl(V_\theta\bigl(s_K^{(b)}\bigr)-R_T^{(b)}
        \Bigr)^2
        $$
        $$
        \mathcal L_{\mathrm{critic}}
        =
        \mathcal L_{\mathrm{M}}+\alpha_T\mathcal L_{\mathrm{C}}
        $$
        \State Update value net
        $
        \theta \gets \theta-\eta_\theta \nabla_\theta \mathcal L_{\mathrm{critic}}
        $
        \State Update raw advantage net
        $
        \psi \gets \psi-\eta_\psi \nabla_\psi \mathcal L_{\mathrm{M}}
        $
        \State Sample $B$ states $\{s_k^{(b)}\}_{b=1}^B$ from $\mathcal R$
        \State Compute actor loss
        $
        \mathcal L_{\mathrm{actor}}
        =
        \frac1B\sum_{b=1}^B
        \bar q_\psi\bigl(s_k^{(b)},\mu_\phi(s_k^{(b)})\bigr)
        $
        \State Update actor by gradient ascent
        $
        \phi \gets \phi+\eta_\phi \nabla_\phi \mathcal L_{\mathrm{actor}}
        $
        \State Soft-update target value net
        $
        \theta^- \gets \tau\theta + (1-\tau)\theta^-
        $
    \EndFor
\EndFor
\end{algorithmic}
\end{algorithm}

 Define
$$
V_N^{\mathrm{sym}}
:=
J_i(\hat p^N,\dots,\hat p^N),
\qquad
V_N^{\mathrm{BR}}
:=
J_i(\hat p^{N,\mathrm{BR}},\hat p^N,\dots,\hat p^N),\qquad \forall\, i\in [N],
$$
and the empirical Nash-equilibrium gap
$$
\mathrm{Gap}_N:=V_N^{\mathrm{BR}}-V_N^{\mathrm{sym}}.
$$
From earlier analysis for the example \ref{example-1}  we know
$
\mathrm{Gap}_N=\cO(N^{-1})$.  Figure~\ref{fig:gap_num}  illustrates the  scaling of $\mathrm{Gap}_N$ against $N$ for $N=2,5,10, 20, 40$.  Figure~\ref{fig:training_curve} plots two training curves for $N=2$ and $N=10$. The numerical result shows 
the scaling of $\mathrm{Gap}_N$ being approximately $O(N^{-1})$.

\begin{figure}
    \centering
    \includegraphics[width=0.4\linewidth]{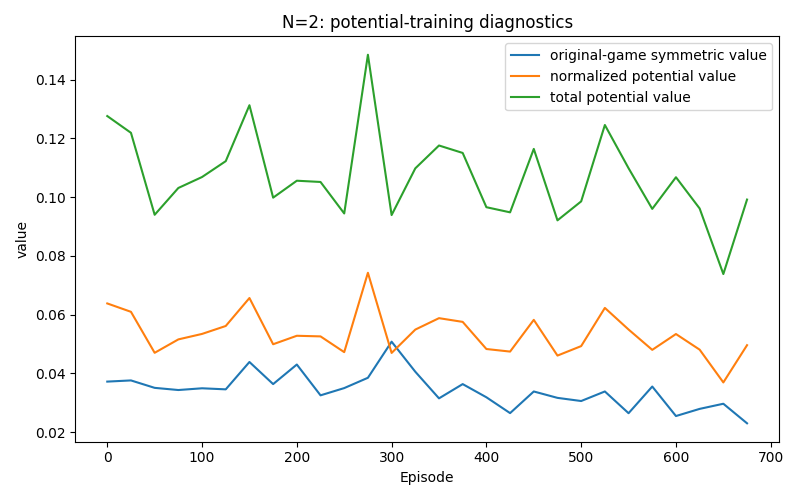}
    \includegraphics[width=0.4\linewidth]{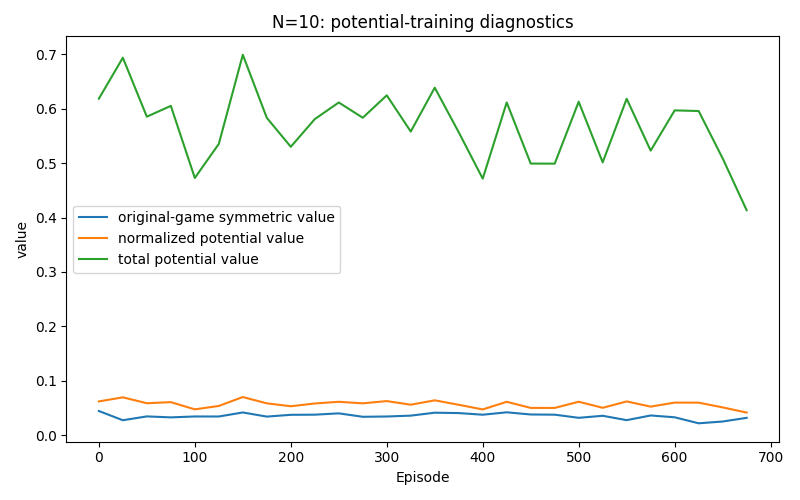}
    \caption{Training curves. Left: $N=2$.   Right: $N=10$.}
    \label{fig:training_curve}
\end{figure}

\begin{figure}
    \centering
    \includegraphics[width=0.4\linewidth]{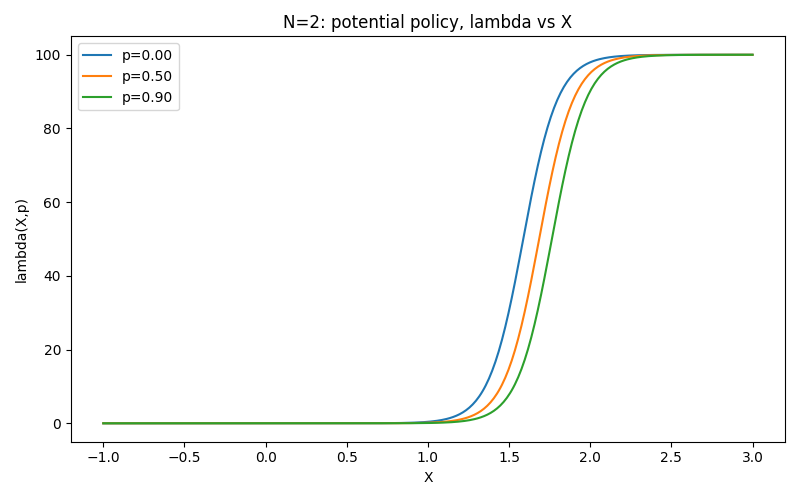}
    \includegraphics[width=0.4\linewidth]{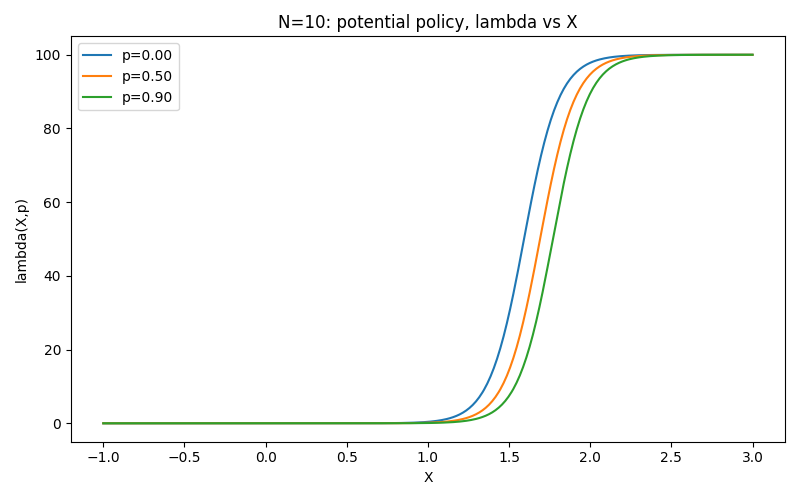}
    \caption{$\lambda$. Left: $N=2$.  Right: $N=10$.}
    \label{fig:p_x}
\end{figure}

\section{Conclusion} This paper studies directly  finite $N$-player games with player-dependent payoff kernels. It first embeds pure stopping times into independently randomized cumulative stopping processes. This convexification preserves pure-profile payoffs and pure Nash equilibria, preserves $\alpha$-potentiality under the canonical randomized extension, and introduces no relaxation gap for potential optimization. For games with local stopped-status interactions,  an $\alpha$-potential function is constructed and its maximization is converted into a multidimensional singular-control problem with local gradient constraints and a separate nonlocal condition for finite jumps. Finally,  a Potential-CT-DDPG algorithm is proposed for learning approximate potential maximizers from simulated trajectories when the model coefficients are unknown.

\bibliography{references}

\end{document}